\documentclass[10pt]{aims}
\usepackage{amsmath}
  \usepackage{paralist}
  \usepackage{graphics} 
  \usepackage{epsfig} 
\usepackage{graphicx}  \usepackage{epstopdf}
 \usepackage[colorlinks=true]{hyperref}
\hypersetup{urlcolor=blue, citecolor=red}
\usepackage{amsfonts,amssymb,stmaryrd,amsthm,mathabx}

\usepackage{subfigure}
\usepackage{times} 
\usepackage{hyperref}
\usepackage{amscd,empheq}
\usepackage{yfonts}
\usepackage{mathrsfs}
\usepackage[cmtip,arrow]{xy}
\usepackage{pb-diagram,pb-xy}
\usepackage{titletoc}
\usepackage{palatino}
\usepackage{bm}
\usepackage{color}
\usepackage{tikz-cd}

\def\currentvolume{X}

   \def\currentmonth{}
    \def\ppages{1-33}
     \def\DOI{}
     
\newtheorem{theorem}{Theorem}[section]
\newtheorem{corollary}[theorem]{Corollary}

\newtheorem{lemma}[theorem]{Lemma}
\newtheorem{proposition}[theorem]{Proposition}

\theoremstyle{definition}
\newtheorem{definition}[theorem]{Definition}
\newtheorem{remark}[theorem]{Remark}
\newtheorem{example}[theorem]{Example}

\title[Lattice Structures for Attractors  III]{Lattice Structures for Attractors  III}

\author[W.D. Kalies, K. Mischaikow, and R.C.A.M. Vandervorst]{}


 \email{mischaik@math.rutgers.edu}
 \email{wkalies@fau.edu}
 \email{vdvorst@few.vu.nl}

\thanks{The first author was partially supported by NSF grant NFS-DMS-0914995
and ARO grant W911NF1810306. The second author was 
partially supported by NSF grants  1521771, 1622401, 1839294, 1841324, 1934924, by NIH-1R01GM126555-01 as part of the Joint DMS/NIGMS Initiative to Support Research at the Interface of the Biological and Mathematical Science and contracts  and DARPA contract HR001117S0003-SD2-FP-011.
The first two  authors are also grateful for  visiting positions at the MBI in Columbus and the CRM in Montreal. }

\newenvironment{fig}
{\begin{figure}[hbt]}
{\end{figure}}
\newcommand{\bfig}{\begin{fig}}
\newcommand{\efig}{\end{fig}}

\newcommand{\rc}{{\scriptscriptstyle\#}}

\newcommand{\setof}[1]{\left\{ {#1}\right\}}

\newcommand{\Hom}{\hbox{\rm Hom}}
\newcommand{\smin}{\smallsetminus}

\newcommand{\R}{{\mathbb{R}}}
\newcommand{\T}{{\mathbb{T}}}
\newcommand{\Z}{{\mathbb{Z}}}

\newcommand{\cA}{{\mathcal A}}
\newcommand{\cB}{{\mathcal B}}
\newcommand{\cC}{{\mathcal C}}
\newcommand{\cD}{\mathcal{D}}

\newcommand{\cF}{{\mathcal F}}

\newcommand{\cM}{{\mathcal M}}

\newcommand{\cR}{{\mathcal R}}
\newcommand{\cS}{{\mathcal S}}

\newcommand{\cU}{{\mathcal U}}
\newcommand{\cV}{{\mathcal V}}

\newcommand{\cX}{{\mathcal X}}

\newcommand{\sA}{{\mathsf A}}
\newcommand{\sB}{{\mathsf B}}

\newcommand{\sE}{{\mathsf E}}
\newcommand{\sF}{{\mathsf F}}

\newcommand{\sI}{{\mathsf I}}
\newcommand{\sJ}{{\mathsf J}}
\newcommand{\sK}{{\mathsf K}}
\newcommand{\sL}{{\mathsf L}}
\newcommand{\sM}{{\mathsf M}}
\newcommand{\sN}{{\mathsf N}}
\newcommand{\sO}{{\mathsf O}}
\newcommand{\sP}{{\mathsf P}}
\newcommand{\sQ}{{\mathsf Q}}

\newcommand{\sT}{{\mathsf T}}
\newcommand{\sU}{{\mathsf U}}

\newcommand{\sfS}{{\mathsf \Sigma}}

\newcommand{\CF}{{\mathsf{C}}}
\newcommand{\CCF}{{\CF^\sigma}}

\newcommand{\sOcl}{{\mathsf O}^{\rm clp}}

\newcommand{\scrR}{{\mathscr R}}

\newcommand{\scrO}{{\mathscr O}}
\newcommand{\scrC}{{\mathscr C}}

\newcommand{\sAtt}{{\mathsf{ Att}}}
\newcommand{\sRep}{{\mathsf{ Rep}}}
\newcommand{\sMorse}{{\mathsf{ Morse}}}
\newcommand{\sMTile}{{\mathsf{ MTile}}}

\newcommand{\sMTileR}{{\mathsf{ MTile}_\scrR}}
\newcommand{\sInvset}{{\mathsf{ Invset}}}

\newcommand{\IS}{{\mathsf{ Invset}}}

\newcommand{\sConvex}{{\mathsf{Co}}}

\newcommand{\sANbhd}{{\mathsf{ ANbhd}}}
\newcommand{\sRNbhd}{{\mathsf{ RNbhd}}}
\newcommand{\sANbhdR}{{\mathsf{ ANbhd}}_{\scrR}}

\newcommand{\sINbhd}{{\mathsf{ INbhd}}}

\newcommand{\sABlockR}{{\mathsf{ABlock}}_{\scrR}}
\newcommand{\sRBlockR}{{\mathsf{RBlock}}_{\scrR}}

\newcommand{\sABlockC}{{\mathsf{ABlock}}_{\scrC}}
\newcommand{\sRBlockO}{{\mathsf{RBlock}}_{\scrO}}

\newcommand{\sIBlock}{{\mathsf{IBlock}}}
\newcommand{\sIBlockR}{{\mathsf{IBlock}}_{\scrR}}

\newcommand{\sIsol}{{\mathsf{ Isol}}}

\newcommand{\sSet}{{\mathsf{Set}}}

\newcommand{\sSCC}{{\mathsf{SC}}}

\newcommand{\bomega}{{\bm{\omega}}}
\newcommand{\balpha}{{\bm{\alpha}}}

\newcommand{\sBDLat}{{\mathbf{BDLat}}}

\newcommand{\sRC}{{\mathsf{RC}}}

\newcommand{\Chi}{\raise .75ex\hbox{$\chi$}}

\newcommand{\pred}[1]{\overleftarrow #1}
\newcommand{\pr}[1]{\overleftarrow #1}

\newcommand{\down}{\downarrow\!}
\newcommand{\up}{\uparrow\!}

\newcommand{\al}{\mathbf{\nu}}

\newcommand{\ji}{\mathbf{\mu}}

\newcommand{\sPoset}{\mathbf{Poset}}

\newcommand{\vgln}{\lb\begin{array}{rcl}}
\newcommand{\eindvgln}{\end{array}\right.}
\newcommand{\alphaOg}{\alpha_{\rm o}(\gamma_x^-)}

\newcommand{\cl}{{\rm cl}\,}

\newcommand{\Int}{\mbox{\rm int\,}}
\newcommand{\Inv}{\mbox{\rm Inv}}

\newcommand{\id}{\mathop{\rm id }\nolimits}

\begin{document}
\begin{sloppypar}
\maketitle

\centerline{\scshape William D. Kalies }
\medskip
{\footnotesize
 \centerline{Florida Atlantic University}
   \centerline{777 Glades Road}
   \centerline{Boca Raton, FL 33431, USA}
} 

\medskip

\centerline{\scshape Konstantin Mischaikow}
\medskip
{\footnotesize
 \centerline{Rutgers University}
   \centerline{110 Frelinghusen Road}
   \centerline{Piscataway, NJ 08854, USA}
}

\medskip

\centerline{\scshape Robert C.A.M. Vandervorst}
\medskip
{\footnotesize
 \centerline{ VU University}
   \centerline{De Boelelaan 1081a}
   \centerline{1081 HV, Amsterdam, The Netherlands}
}

\bigskip


\begin{abstract}
The theory of bounded, distributive lattices provides the appropriate language for describing directionality and asymptotics in dynamical systems.
For bounded, distributive lattices the general notion of `set-difference' taking values in a semilattice is introduced, and is called the Conley form.
The {\em Conley form} is used to build concrete, set-theoretical models of spectral, or Priestley spaces, of bounded, distributive lattices and their finite
coarsenings.
Such representations build order-theoretic models of dynamical systems, which
are used to develop tools for computing global characteristics of a dynamical system.
\end{abstract}

\section{Introduction}
\label{sec:intro}

The global structure of a nonlinear dynamical system can be characterized in terms of its recurrent and nonrecurrent dynamics.
A systematic approach to this decomposition began with Smale \cite{Smale} in the context of Axiom A diffeomorphisms.
For general dynamical systems, Conley \cite{Conley} established the concept of a Morse  decomposition that dynamically defines an order on a finite collection of invariant sets that contain the recurrent dynamics, cf.\ Definition \ref{defn:MD}, whereby finite posets are introduced as a description of nonrecurrent  global dynamics.

To fully analyze the dynamical information associated with a Morse decomposition, one needs the existence of an index filtration to obtain a connection matrix~\cite{Franzosa2,Robbin:1992wp,Harker:Mischaikow:Spendlove}. As observed by Robbin and Salamon~\cite{Robbin:1992wp}, an index filtration is a finite lattice of attracting blocks. Indeed, the set of all  attracting blocks as well as the set of all attractors in a dynamical system  have the structure of a bounded, distributive lattice.

The duality between posets that capture the gradient-like nature of the dynamics and lattices that provide insight into the global organization of the dynamics is explored in a series of papers \cite{KMV-0,KMV-1a,KMV-1b,KastiKV} that develop 
an algebraic representation of the nonrecurrent structure of nonlinear dynamics.
While this effort has intrinsic mathematical interest, it also  has considerable practical value.
For the past 25 years there has been a systematic attempt to exploit ideas from Conley theory in the context of rigorous numerical analysis of nonlinear systems \cite{MMLorenz,DJM,MZ,SH,CH,THPP} and  as a technique for the analysis of time-series data \cite{mischaikow:mrozek:reiss:szymczak, cummins:gedeon:harker:mischaikow,DKP}.
In this setting the natural formalism is that of a relation $\cF\subset \cX \times \cX$ where $\cX$ is a finite set that represents a discretization of the phase space, and the dynamics is generated by iterations of $\cF$.
Fundamental questions are: (1) how well does the numerical approximation $\cF$ capture the dynamics of the original system $\varphi$, and (2) having generated $\cF$ from data, what models $\varphi$ produce dynamics compatible with $\cF$.
As is made clear below, the algebra derived in this paper  provides a framework in which  to address these questions.

The above mentioned work \cite{KMV-0,KMV-1a,KMV-1b,KastiKV} focuses on the properties of lattices of attractors/repellers and attracting/repelling blocks for two rather general models of dynamics:
the \emph{combinatorial dynamics} generated by a relation $\cF$ defined on a finite set $\cX$ as indicated above, and single-valued dynamics defined on a compact metric space $X$ as a continuous function $\varphi\colon \T^+ \times X \to X$, cf.\ App.\ \ref{sec:dynamics}.

Directionality  and asymptotics associated with time in these systems can be expressed by the duality between the lattices of attractors and repellers. 
In particular, recall from \cite{KMV-1a,KMV-1b} the following two commutative diagrams.
First, in the context of dynamical systems
\begin{equation}\label{diag:AR}
\begin{diagram}
\node{\sABlockC(\varphi)} \arrow{e,l,<>}{^c} \arrow{s,l,A}{ {\omega}}\node{\sRBlockO(\varphi)} \arrow{s,r,A}{ {\alpha}}\\
\node{\sAtt(\varphi)} \arrow{e,l,<>}{^*}   \node{\sRep(\varphi)} 
\end{diagram}
\end{equation}
where $\sABlockC(\varphi)$ and $\sRBlockO(\varphi)$ are the lattices of closed attracting/open repelling blocks and $\sAtt(\varphi)$ and
$\sRep(\varphi)$ are the lattices of attractor and repellers. 
The  duality mappings, $^c$  and $^*$, are involutions, and the latter is defined as $A\mapsto A^*$, the dual repeller of $A$, cf.\ App.\ \ref{sec:topdyn}.

Second, for combinatorial systems the corresponding diagram is 
\begin{equation}
\begin{diagram}
\node{\sInvset^+(\cF)} \arrow{e,l,<>}{^c}\arrow{s,l,A}{\bomega} \node{\sInvset^-(\cF)} \arrow{s,r,A}{\balpha} \\
\node{\sAtt(\cF)}     \arrow{e,l,<>}{*} \node{\sRep(\cF)} 
\end{diagram}
\end{equation}
where $\sInvset^+(\cF)$ and $\sInvset^-(\cF)$ are the lattices of forward and backward invariant sets, respectively, and $\sAtt(\cF)$ and $\sRep(\cF)$ are the lattices of attractor and repellers.
Again, the duality mappings $^c$ and $^*$ are involutions where the latter is defined by $\cA \mapsto \cA^*:=\balpha(\cA^c)$, cf.\ App.\ \ref{sec:combdyn}.

Of course, neither of these diagrams directly provides the poset structure of a Morse decomposition that explicitly describes an order in the gradient-like dynamics.
The missing ingredient  is a notion of set-difference within the setting of lattice theory. 
In a Boolean algebra $(\sB,\vee,\wedge,\urcorner,0,1)$ the derived operation {\em set-difference} on $\sB$ is given by $a,b\mapsto a \smin b =a\wedge b\urcorner$.
We construct the natural analogue for bounded, distributive lattices. 

In particular, in Section~\ref{sec:MPbdL}, given a bounded, distributive lattice $\sL$ we define a notion of set-difference, called the \emph{canonical Conley} form, via $\sB(\sL)$ which is the  Booleanization of $\sL$, cf.\ Section~\ref{sec:booldual}.
There is a natural embedding $j\colon \sL\to\sB(\sL)$, and since $\sB(\sL)$ is a Boolean algebra, we define the canonical Conley form to be 
\[
(a,b) \mapsto \CCF(a,b) := j(a)\smin j(b) = A\smin B \in \sB^\updownarrow(\sL), 
\]
where $\sB^\updownarrow(\sL)$ is the semilattice generated by the canonical Conley form.

However, abstract knowledge of the existence of $\CCF$ and $\sB^\updownarrow(\sL)$ is of limited value.
In Section~\ref{CFdefn} we introduce the \emph{Conley forms on $\sL$ in $\sI$}.
These are  semilattice morphisms $\CF\colon \sL \times \sL \to \sI$ where $\sI$ is an explicit meet semilattice consisting of structures of interest,
 and $\CF = \gamma \circ \CCF$ for some injective semilattice homomorphism $\gamma$.
Remarkably, any Conley form on $\sL$ is characterized by the following three properties:
\begin{enumerate}
\item[(Absorption)] $\CF(a\vee b, a) = \CF(b,a)$  and  $\CF(a,a \wedge b) = \CF(a,b)$ for all $a,b\in \sL$.
\item[(Distributivity)]  \; $\CF(a\wedge c,b\vee d) = \CF(a,b) \wedge \CF(c,d)$ for all $a,b,c,d\in \sL$;
\item[(Monotonicity)]  \; $\CF(a,b)=\CF(0,1)$ implies $a\le b$ for $a,b\in\sL$.
\end{enumerate}
This leads to the following result, cf.\ Theorem~\ref{mainthm12}.
\begin{theorem}
\label{main1}
Let $\CF\colon \sL \times \sL \to \sI$ and $\CF'\colon \sL \times \sL \to \sI'$ be Conley forms. Then, there exists a meet semilattice isomorphism
$g\colon \CF(\sL\times\sL) \to \CF'(\sL\times \sL)$ such that $\CF'=g\circ \CF$.
\end{theorem}

In particular, $\CCF$ is a Conley form.
Since  Conley forms are unique  up to isomorphisms, we adopt the notation 
\[
a-b := \CF(a,b),
\]
if there is no ambiguity about the specific representation. 

Observe that $\sL$ need not be a finite lattice. Indeed lattices of attactors and attracting blocks are often infinite. The proof of the Theorem~\ref{main1} relies on the compactness of the spectrum $\sfS(\sL)$ in the Priestley topology, cf.~Section~\ref{sec:booldual}.

As is demonstrated in Section~\ref{exoflf}, with the Conley form we are able to identify \emph{Morse sets} from the lattice structures of attractors and repellers.
In particular, in Example~\ref{ex:X4} we use the Conley form on $\sAtt(\varphi)$ in $\sInvset(\varphi)$, the  lattice  of invariant sets, to define $\sMorse(\varphi)$, the semilattice of Morse sets of $\varphi$.
More precisely,
\[
\sMorse(\varphi):=\CF_\sAtt(\sAtt(\varphi)\times\sAtt(\varphi))
\]
where $\CF_\sAtt(A,A') = A-A' := A\cap A'^*$, cf.\ \cite[II.5.3.E]{Conley}.
Example~\ref{CFcomb} uses the same formalism to define Morse sets for the dynamics generated by a relation $\cF$.

As is shown in Section~\ref{mapslf}, homomorphisms between bounded distributive lattices lead to homomorphisms between  Conley forms. 
This provides us with a tool to analyze the global structure of invariant sets.
For example, if $\sINbhd(\varphi)$ and $\sIsol(\varphi)$ are the meet semilattices of isolating neighborhoods and isolated invariant sets respectively, cf.\ Section\ \ref{dynsys12}, then the following diagram shows how the Conley forms on $\sABlockC(\varphi)$ and $\sAtt(\varphi)$  define isolated invariant sets
\begin{equation}\label{diag:AR2}
\begin{diagram}
\node{\sABlockC(\varphi)\times \sABlockC(\varphi)} \arrow{e,l}{\CF } \arrow{s,l,A}{  \omega\times \omega}\node{\sINbhd(\varphi)} \arrow{s,r,A}{ \Small \Inv}\\
\node{\sAtt(\varphi)\times \sAtt(\varphi)} \arrow{e,l}{\CF_\sAtt }   \node{\sIsol(\varphi)} 
\end{diagram}
\end{equation}
where  
\[
\CF(U,U') = U\cap U'^c = U\smallsetminus U'\;\; \hbox{for $U,U'\in \sABlockC(\varphi)$.}
\]

The remainder of the paper uses the tools developed in Sections~\ref{sec:booldual}-\ref{CFdefn} and \ref{mapslf} to provide algebraic representations of structures of global dynamics.
In Section~\ref{sec:attmd} we provide partitions of phase space, called \emph{Morse tiles}, in the context of continuous and combinatorial dynamics.
Furthermore, we discuss Morse tiles in the context of closed regular sets as these provide a useful computational structure.

In Section~\ref{sec:MRandMD} we turn to the goal mentioned earlier in this introduction: an explicit description of the relationship between the order relations on Morse  decompositions and the lattice structures of attractors and repellers.

\begin{definition}
\label{defn:MD}
Given a dynamical system $\varphi\colon \T^+\times X\to X$ on a compact metric space, a \emph{Morse representation} of $\varphi$ is a finite poset $(\sM,\le)$ where $\sM$  consists of mutually disjoint, nonempty, 
compact invariant sets, called \emph{Morse sets}, with the property that for each $x\in X$ there exists $M\in\sM$ such that $\omega(x)\subset M$,
and for each complete orbit $\gamma_x$ with $x
\notin\bigcup_{M\in\sM}M$ there exist $M<M'$ such that $\omega(x)\subset M$ and $\alphaOg\subset M'$.
\end{definition}

As is discussed in Section~\ref{subsec:repr1}, every Morse representation can be generated from  a finite sublattice of attractors $\sA$ with  the associated Morse representation given by 
\[
\sM(\sA) = \setof{\CF_\sAtt(A,\pred{A})\mid A\in\sJ(\sA)}
\]
where $\sJ(\sA)$ denotes the set of join-irreducible elements of $\sA$, and $\pred{A}$ is the unique immediate predecessor of $A$.
Given a finite sublattice $\sN$ of attracting blocks, and the surjective homomorphism $\omega\colon \sN\twoheadrightarrow \sA$, we obtain a dual order-embedding $\pi\colon \sM(\sA) \hookrightarrow \sT(\sN)$ where 
$$
\sT(\sN)=\{\CF^b(N,\pred{N})=N\smallsetminus N'~|~N\in\sJ(\sN)\}
$$
via the Conley form on $\sN$. The map $\pi\colon \sM(\sA) \hookrightarrow \sT(\sN)$ is dual to $\omega\colon \sN\twoheadrightarrow \sA$ and is referred to as a
 \emph{tesselated Morse decomposition}, cf.\ Theorem \ref{thm:mainMDthm}.

As is mentioned earlier, the ideas from Conley theory are being used in the context of rigorous computations
and data analysis, and thus a fundamental question is how does the dynamics captured by a relation $\cF$ compare to the dynamics of a continuous system $\varphi$?
We address this question in Section~\ref{subsec:repr2}.
Closed regular sets, e.g.\ triangulations or regular CW-complexes, provide a wide variety of discretizations of phase space for continuous dynamical systems, and as is shown in Section~\ref{sec:regularClosed}, is rich enough to capture the lattice of attractors of a continuous system $\varphi$. 
This leads us to consider the span
\begin{equation}
\label{span11}
\begin{diagram}
\node{\scrR(X)} \node{\sABlockR(\varphi)}\arrow{w,l,V}{\supset}\arrow{e,l,A}{\omega} \node{\sAtt(\varphi),}
\end{diagram}
\end{equation}
where $\scrR(X)$ are the regular closed sets in $X$ and $\sABlockR(\varphi)\subset \sABlockC(\varphi)$ are the regular closed attracting blocks for $\varphi$, cf.\ App.\ \ref{sec:dynamics}.
Let $\scrR_0$ be a finite subalgebra of $\scrR(X)$.
Let $\cX$ be an indexing set for the atoms of $\scrR_0$ and 
 $|\cdot |\colon \sSet(\cX)\to \scrR_0\subset \scrR(X)$
by identifying sets of labels with regular closed sets by unioning them.
Finally, consider a relation $\cF\subset \cX\times\cX$.
Ideally, we have the existence of the following commutative diagram, that we refer to as a \emph{commutative combinatorial model}  for $\varphi$: 
\begin{equation}
\label{model2a2}
\begin{diagram}
\node{\scrR(X)}\node{\sABlockR(\varphi)} \arrow{w,l,V}{\subset}\arrow{e,l,A}{\omega} \node{\sAtt(\varphi)} \\
\node{\sSet(\cX)}\arrow{n,l}{|\cdot |} \node{\IS^+(\cF)}\arrow{n,l}{|\cdot |}\arrow{w,l,V}{\subset}\arrow{e,l,A}{\bomega} \node{\sAtt(\cF),}\arrow{n,l}{\omega(|\cdot |)}
\end{diagram}
\end{equation}
Theorem~\ref{critWL} provides an exact characterization of the properties of $\cF$ such that \eqref{model2a2} commutes.

The above description takes the perspective that $\varphi$ is the object of primary importance and $\cF$ is derived in order to study the dynamics of $\varphi$ computationally. 
However, if one begins with data, then there are a variety of methods by which one can derive a relation $\cF$.
In this setting Theorem~\ref{critWL} provides constraints on continuous models $\varphi$ that are compatible with the data.
An open problem, but of increasing relevance in an age of data driven science, is to derive techniques for choices of maps or differential equations that generate $\varphi$.

We conclude by noting that we have restricted our attention in this paper to single-valued, continuous dynamical systems and to combinatorial dynamical systems.
There are, of course, other models for continuous dynamics, e.g.\ set-valued \cite{aubin:frankowska,akin,mcgehee,arnold}, and for combinatorial dynamics, e.g.\ combinatorial vector fields \cite{kaczynski:mrozek:wanner, mrozek:2017}.
It is our belief that the algebraic structures developed in this paper can be applied equally well in these other settings.

\section{Booleanization}
\label{sec:booldual}
In this section we describe two  algebraic principles,
Booleanization and duality. These tools  are fundamental   to the description of the algebraic structures of global dynamics. 
Denote the categories of bounded, distributive lattices and posets by  $\sBDLat$ and $\sPoset$ respectively.
A bounded, distributive lattice has unique neutral elements $0$ and $1$, and
in $\sBDLat$  all lattice homomorphisms preserve  $0$ and $1$, and all sublattices contain $0$ and $1$.

There are two functors that relate $\sBDLat$ and $\sPoset$.
The \emph{down-set functor} $\sO\colon \sPoset \Rightarrow \sBDLat$ is a contravariant functor that assigns  to a poset $\sP$ the bounded, distributive lattice of down-sets denoted by $(\sO(\sP),\cup,\cap)$. 
Recall that a {\em down-set} $I$ in a poset $\sP$ is defined via the property that $p\in I$ and $q\le p$ implies $q \in I$.
The \emph{spectral functor} $\sfS\colon \sBDLat\Rightarrow \sPoset$ is a contravariant functor that assigns to a bounded, distributive lattice $\sL$ the poset $\bigl(\sfS(\sL),\subset\bigr)$ of the prime ideals in $\sL$ called the \emph{spectrum} of $\sL$.
Recall that an {\em ideal} in a bounded, distributive lattice is a down-set $I$ that is closed under join, i.e.~$a,b\in I$ implies $a\vee b\in I$.
An ideal $I$ is a {\em prime ideal} if $a\wedge b\in I$ implies $a\in I$ or $b\in I$.
The prime ideals are exactly the pre-images $I=f^{-1}(0)$ with $f\in \Hom(\sL,{\mathbf 2})$ where ${\mathbf 2}$
denotes the lattice of two elements 
$\{0,1\}$, cf.\ \cite{Roman,Davey}.

A classical result due to Birkhoff states that a bounded, distributive lattice is isomorphic to a sublattice of $\sSet(\sfS(\sL))$.
The map
\begin{equation*}
\begin{aligned}
j\colon \sL & \to \sO(\sfS(\sL)) \\
a &\mapsto j(a) =  \{I\in \sfS(\sL)~|~a\not \in I\}
\end{aligned}
\end{equation*}
defines such an embedding.
The map $j$ is not surjective in general. 
However, when $\sL$ is finite, it is surjective, and this fact is called the Birkhoff Representation Theorem for finite, distributive lattices \cite[Theorem~6.6]{Roman}. 
In the case that $\sL$ is a Boolean algebra, Stone introduced a topology on the spectrum in order to characterize the image of $j$, and this characterization is known as the Stone Representation Theorem \cite[Theorem~10.18]{Roman}. 
The idea underlying the Stone representation is
that since the clopen sets in a topological space form a Boolean algebra, one can topologize $\sfS(\sL)$ so that the
image of $j$ is the algebra of clopen sets, cf.\  \cite{Roman,Davey}.

For bounded, distributive lattices, Priestley introduced a topology on the spectrum that
determines the image of $j$. 
Priestley's topology  is induced by the basis 
\[
 \{j(a) \smin j(b)~|~a,b\in \sL\},
\]
where $j(a)\smin j(b) := j(a) \cap j(b)^c$ is set-difference.
Since $j(a), j(b)^c$ are basic open sets, by choosing 
$b, a=\varnothing$, all the basic open sets  are also closed and thus clopen. Note that each $j(a)$ is a down set so that the  image of $j$ is 
a sublattice of the down sets of $\sfS(\sL)$.
The Priestley Representation Theorem characterizes the
image of $j$ as the clopen down sets of $\sfS(\sL)$
\[
\sOcl(\sfS(\sL)) = \sB^\downarrow(\sL) := \{j(a)~|~a \in \sL\},
\]
and $\sL$ is isomorphic to $\sB^\downarrow(\sL)$ via the map $j\colon\sL\to\sB^\downarrow(\sL)$, cf.\ \cite{Roman,Davey}.

The spectrum $(\sfS(\sL),\subseteq)$ is a poset and with the Priestley topology the spectrum  is a compact and totally order-separated topological space, called a \emph{Priestley space}. Priestley spaces are necessarily Hausdorff and $0$-dimensional.
The Priestley Representation Theorem states 
that the category of bounded, distributive lattices is  dually equivalent to the category of Priestley spaces.

Birkhoff's theorem that every bounded, distributive lattice  is isomorphic to a sublattice of the Boolean algebra $\sSet(\sfS(\sL))$ motivates the question of obtaining a  smallest Boolean algebra in which the lattice embeds. Such a Boolean algebra is called a \emph{Booleanization}, or \emph{free Boolean extension}, and
a general procedure to obtain a specific Booleanization is based on the Priestley Representation Theorem \cite[Theorem~10.15]{Roman}, cf.\ \cite{Balbes,Miraglia,Vickers}.

Let $\sB(\sL)$ be the Boolean algebra of all clopen subsets of $\sfS(\sL)$. 
The above construction yields
the following Booleanization theorem.

\begin{proposition}[Theorem~10.19 in \cite{Roman}]
\label{prop:booleanize2}
For every bounded distributive lattice $\sL$, 
the map
\begin{equation}\label{jmap}
\begin{aligned}
j\colon \sL & \to \sB(\sL) \\
a &\mapsto j(a) =  \{I\in \sfS(\sL)~|~a\not \in I\}
\end{aligned}
\end{equation}
is the unique lattice monomorphism
 with the property that for every homomorphism $h\colon \sL\to \sE$ to a Boolean algebra $\sE$  
there exists a unique lattice homomorphism $\sB(h):\sB(\sL)\to \sE$ such that $\sB(h)\circ j=h$. The Boolean algebra $\sB(\sL)$ is called the {Booleanization} of $\sL$ and the  mapping $\sB(h)\colon \sB(\sL) \to \sE$ is Boolean.
\end{proposition} 

\begin{remark}
Throughout the rest of this paper  
\[
j\colon \sL \to \sB(\sL)
\]
denotes the specific lattice monomorphism of Proposition~\ref{prop:booleanize2}. 
Furthermore, when we are explicitly working with this monomorphism given $a\in \sL$ we denote its image in $\sB(\sL)$ by $A$ so that $A=j(a)$.
\end{remark}

The Booleanization theorem above also applies to homomorphisms $h\colon\sK\to \sL$. Proposition \ref{prop:booleanize2} yields the following commutative diagram
\begin{equation}
\begin{diagram}
\label{commBool}
\node{\sK}\arrow{s,l}{j} \arrow{e,l}{h}\node{\sL}\arrow{s,r}{j}\\
\node{\sB(\sK)} \arrow{e,l}{\sB(h)}\node{\sB(\sL)}
\end{diagram}
\end{equation}
In particular $j\bigl(h(a)\bigr) = \sB(h)(A)$ where $A=j(a)$.

Booleanization is a (covariant) functor, and $\sB$ is obtained via the composition
$\sOcl\circ \sF\circ \sfS$, where $\sOcl$ is the clopen down-set functor, $\sF$ is the functor which
removes all order relations to produce the trivial order,
and $\sfS$ is the spectral functor. Moreover,
due to the compactness of the Priestley topology, each element of $\sB(\sL)$ can be written as a finite union of the  clopen convex sets
\[
\sB^\updownarrow(\sL) := \{ A \smin B~|~ A,B\in \sB^\downarrow(\sL)\},
\]
 cf.\ 
 \cite[Theorem~10.10]{Roman} and   \cite[Lemma~11.22]{Davey}.

For finite, distributive lattices the spectrum $\sfS(\sL)$ is order-isomorphic to the poset of join-irreducible elements $\sJ(\sL)$. A nonzero element $a\in\sL$ is {\em join-irreducible} if $a$ has a unique predecessor in $\sL$ which is denoted by $\pred{a}$. The order-isomorphism
$\sJ(\sL)\to\sfS(\sL)$ is given by the map
$
a\mapsto (\up a)^c.
$
Every element in $\sL$ can be written as a join of join-irreducible elements, for example
\[
a=\bigvee_{a'\le a\atop a'\in \sJ(\sL)} a'.
\]
Such join-representations are not unique, but each element has a unique irredundant join-representation, cf.\ \cite[Thm.\ 4.30]{Roman}.

\section{Lattice forms} 
\label{sec:MPbdL} 
In general, lattices do not allow complements, which makes it impossible to define the analogues of convex sets directly.
A first step towards defining
 an analogue of the set difference operation on bounded, distributive lattices
 are lattice forms.  

Given a lattice $\sL$ the set $\sL\times\sL$ has a natural meet semilattice structure defined by $(a,b)\wedge (c,d):= (a\wedge c,b\vee d)$, with neutral elements $0=(0,1)$ and $1=(1,0)$. 
It follows that 
\begin{equation}\label{eqn:order}
(a,b)\le (c,d) \hbox{ if and only if } a\le c \hbox{ and } b\ge d.
\end{equation}

\begin{definition}
\label{lem:mp11}
Let $\sL$ be a lattice, and let $\sI$ be a meet semilattice.
A  {\em lattice form on $\sL$ represented in $\sI$} is a function $\rho\colon\sL\times\sL \to \sI$ satisfying the property
\begin{enumerate}
\item[(Absorption)] $\rho(a\vee b, a) = \rho(b,a)$  and  $\rho(a,a \wedge b) = \rho(a,b)$ for all $a,b\in \sL$.
\end{enumerate}
\end{definition}

\begin{example}
\label{ex:X1}
If $\sL$ is a Boolean algebra, then $(a,b) \mapsto \rho(a,b)=a\smin b:=a\cap b^c$ defines a  lattice form represented in $\sL$.
This lattice form also satisfies the following properties
\begin{enumerate}
\item[(Distributivity)]  \; $\rho(a\wedge c,b\vee d) = \rho(a,b) \wedge \rho(c,d)$ for all $a,b,c,d\in \sL$;
\item[(Monotonicity)]  \; $\rho(a,b)=\rho(0,1)$ implies $a\le b$ for $a,b\in\sL$,
\end{enumerate}
which are called \emph{distributive} and \emph{monotone} lattice forms respectively.
A concrete example is the Boolean algebra consisting of subsets of a set $X$ denoted by $(\sSet(X),\cup,\cap,\,^c,\varnothing,X)$.
\end{example}

\begin{lemma}
\label{exchange}
From the distributivity property the following exchange property follows 
\begin{enumerate}
    \item[{\rm (Exchange)}] \: $\rho(a,b) \wedge \rho(c,d)
=\rho(a,d) \wedge \rho(c,b)$ for all $a,b,c,d\in \sL$.
\end{enumerate}
\end{lemma}

\begin{proof}
From distributivity we have that 
\[
\rho(a,b) \wedge \rho(c,d) = \rho(a\wedge c,b\vee d) =\rho(a,d) \wedge \rho(c,b),
\]
which proves the lemma.
\end{proof}

\begin{proposition}\label{prop:LL}
If $\rho\colon\sL\times\sL\to \sI$ is a distributive lattice form, then $\rho$ is a meet semilattice homomorphism.
\end{proposition}
\begin{proof}
By distributivity
\[
\rho((a,b)\wedge(c,d))=\rho(a\wedge c,b\vee d)=\rho(a,b)\wedge\rho(c,d),
\]
which proves that $\rho$ preserves meet operations.
\end{proof}

The following proposition lists a number of properties of distributive lattice forms.
\begin{proposition}
\label{lem:extraprop12}
Let $\rho\colon\sL\times \sL \to \sI$ be a distributive lattice form. 
Then, 
\begin{enumerate}
\item [\emph{(i)}] $\rho(a,b)\le \rho(c,d)$ for all $a\le c$ and $b\ge d$;
\item [\emph{(ii)}]  $\rho(0,1) \le \rho(a,b)\le \rho(1,0)$ for all $a,b\in \sL$;
\item [\emph{(iii)}] $\rho(0,a) =\rho(0,1)$ and $\rho(a,1)=\rho(0,1)$ for all $a\in \sL$.
\end{enumerate}
If in addition $\rho$ is monotone, then
\begin{enumerate}
\item [\emph{(iv)}] $\rho(a,b)=\rho(0,1)$ if and only if $a\le b$.
\end{enumerate}
\end{proposition}

\begin{proof}
(i) From distributivity it follows that $\rho$ is a meet semilattice homomorphism and thus order-preserving. The order on  $\sL\times \sL$ is given by \eqref{eqn:order}.

(ii) Apply (i) to the inequalities $0\le a \le 1$ and $1\ge b \ge 0$.

(iii) From absorption we have that $\rho(1,1)=\rho(1\vee 0,1)=\rho(0,1)$. Then by (i) and (ii) and absorption we have that 
\[
\rho(0,1) \le \rho(0,a) \le \rho(a,a) = \rho(a,a\wedge 1) = \rho(a,1) \le \rho(1,1) = \rho(0,1).
\]

(iv)  By (i) and (ii), 
$\rho(0,1) \le \rho(a,b) \le \rho(a,a) = \rho(0,1)$,
which shows that $\rho(a,b) = \rho(0,1)$.
The other direction is monotonicity
completing the proof.
\end{proof}

By Property (ii) in Proposition \ref{lem:extraprop12} the elements $0:=\rho(0,1)$ 
and $1:=\rho(1,0)$ are the  neutral elements in the meet semilattice $\rho(\sL\times\sL) \subset \sI$. However, in general
the semilattice $\sI$ need not have neutral elements, and if there are neutral elements they need not coincide with $(0,1)$ and $(1,0)$.

\begin{proposition}
\label{prop:restrictionMP}
Let  $\rho\colon \sL\times \sL\to \sI$ be a lattice form, and let $\sK\subset \sL$ be a sublattice.
Then, the form $\rho|_\sK\colon\sK\times\sK\to \sI$
defined by restriction
is a   lattice form on $\sK$. The properties of distributivity and monotonicity are also preserved under restriction.
\end{proposition}

\begin{proof}
Absorption, distributivity, and monotonicity follow from the fact that $\sK$ is a sublattice $\sL$.
\end{proof}

\begin{definition}
\label{defn:canonicalCF}
Let $\sL$ be bounded, distributive lattice.  The \emph{canonical Conley form} on $\sL$ is defined by
\begin{equation}
\label{eqn:canMPa}
\begin{aligned}
\CCF\colon \sL\times\sL & \twoheadrightarrow \sB^\updownarrow(\sL) \\
(a,b) &\mapsto \CCF(a,b) := A\smin B, 
\end{aligned}
\end{equation}
where $A=j(a)$ and $B = j(b)$ and $j\colon \sL \to \sB(\sL)$ is as defined in Proposition~\ref{prop:booleanize2}.
\end{definition}

\begin{proposition}
\label{prop:mp11}
For a bounded, distributive lattice $\sL$ the canonical Conley form is a monotone, distributive lattice form.
\end{proposition}

\begin{proof}
We first prove the absorption property.
Observe that
\[
\begin{aligned}
\CCF(a\vee b,a) & = (A\cup B)\smin A = B\smin A 
 = \CCF(b,a)
\end{aligned}
\]
and
\[
\begin{aligned}
\CCF(a,a\wedge b) & = A \smin (A\cap B) = A\smin B 
 = \CCF(a,b)
\end{aligned}
\]
As for distributivity and monotonicity we have: 
\[
\begin{aligned}
\CCF(a\wedge c, b\vee d) & = (A\cap C) \smin (B\cup D) = (A\smin B)\cap (C\smin D)
 = \CCF(a,b) \cap \CCF( c,d )
\end{aligned}
\]
so that distributivity is satisfied.
Observe that $\CCF(a,b)=A\smin B = \varnothing$ implies that $A\subseteq B$.
Since $j$ is a lattice monomorphism, we conclude that $a\leq b$. Hence, monotonicity is satisfied.
\end{proof}

\section{The Conley form on bounded, distributive lattices}
\label{CFdefn}

The canonical Conley form on a bounded distributive lattice $\sL$ takes values in $\sB^\updownarrow(\sL)$, an abstractly defined semilattice.
For applications it is desirable to represent this form in particular meet semilattices.
With this in mind let $\sI$ be a meet semilattice, and   $\gamma\colon \sB^\updownarrow(\sL)\rightarrowtail \sI$ be a meet injective semilattice homomorphism.
Define 
\begin{equation}
\label{eqn:canMPb}
\begin{aligned}
\CF\colon\sL\times\sL &\to \sI\\
(a,b)&\mapsto \CF(a,b) := \gamma(A\smin B)     
\end{aligned}
\end{equation}
and set $\sI_\CF := \gamma(\sB^\updownarrow(\sL))$.
Observe that since $\gamma$ is injective, $\gamma\colon \sB^\updownarrow(\sL)\to \sI_\CF$ is an isomorphism.

\begin{lemma}
\label{lemma:mp12as}
$\CF$ is a  monotone, distributive lattice form. 
\end{lemma}

\begin{proof}
Since the canonical Conley form is distributive and monotone, the properties are transferred to $\CF$ under the injection $\gamma$ as in the following diagram

\begin{equation}
\label{eqn:AR33a}
\dgARROWLENGTH=3em
\begin{diagram}
\node{\sB^\downarrow(\sL)\times\sB^\downarrow(\sL)} \arrow{e,l,A}{\sB(\CCF)} \node{\sB^\updownarrow(\sL)} \arrow{s,lr}{\gamma }{\cong}\arrow{se,l,V}{\gamma}\\
\node{\sL\times \sL} \arrow{n,lr}{j\times j}{\cong}\arrow{e,l,A}{\CF} \arrow{ne,l,A}{\CCF} \node{\sI_\CF} \arrow{e,l,V}{}\node{\sI}
\end{diagram}
\end{equation}
where $\sB(\CCF)$ restricted to $\sB^\downarrow(\sL) \times \sB^\downarrow(\sL)$ is given by $(j(a),j(b))\mapsto j(a) \smin j(b)$ and is the Booleanization of $\CCF$ via the composition $\sL\times \sL 
\twoheadrightarrow \sB^\updownarrow(\sL) \rightarrowtail \sB(\sL)$.
\end{proof}

We now turn to the main result of this section that characterizes
monotone, distributive lattice forms as representations of the 
canonical Conley form in a given meet semilattice.

\begin{theorem}
\label{mainthm12}
Let $\sL$ be a bounded, distributive lattice and let $\sI$ be a meet semilattice. 
If $\gamma\colon\sB^\updownarrow(\sL)\to\sI$ is a meet injective semilattice homomorphism, then $\CF=\gamma\circ\CCF$ is a monotone, distributive lattice form. 
Conversely, if $\CF\colon\sL\times\sL\to\sI$ is a monotone, distributive lattice form, then there exists an injective meet semilattice homomorphism $\gamma\colon\sB^\updownarrow(\sL)\to\sI$ defined by
\begin{equation}
\label{twostepequiv3}
\gamma(A\smin B) := \CF(a,b),
\end{equation}
such that $\CF=\gamma\circ\CCF$.
\end{theorem}

\begin{proof}
Combining Lemma~\ref{lemma:mp12as} with  Lemmas~\ref{twostepequiv1} and~\ref{twostepequiv2} below proves the theorem.
\end{proof}

\begin{lemma}
\label{twostepequiv1}
Let $\rho\colon\sL\times\sL \to \sI$ be a lattice form. 
Then, $\gamma\colon\sB^\updownarrow(\sL)\to \sI$ given by 
\[
\gamma(A\smin B):=\rho(a,b)
\]
is a well-defined function. 
\end{lemma}

\begin{proof}
We need to prove that if 
$A\smin B = A'\smin B'$, then $\rho(a,b) = \rho(a',b')$.
Observe that, 
since $\rho$ is a lattice form, we have $\rho(a,b) = \rho(a,a\wedge b)$
and thus we may assume without loss of generality, by possibly replacing $b$ by $a\wedge b$, that $b\le a$. The same holds for
$b'\le a'$.
Since $e:=A\smin B = A'\smin B'$ we have
\[
\begin{aligned}
(A\cup A')\smin (B\cup B') &= (A\cup A')\cap (B^c\cap B'^c)\\
&= \bigl((A\smin B) \cap B'^c \bigr)\cup \bigl((A'\smin B') \cap B^c \bigr)\\
&= (e\smin B') \cup (e\smin B) = e\cup e = e.
\end{aligned}
\]
Therefore assume without loss of generality that $B\subset A\subset A'$ and $B\subset B'\subset A'$. 
Since $A'=e\cup B'$ and $ A = e\cup A$, we have that
 $A\cup B' = e\cup A\cup B' = A'\cup A = A'$.
 Similarly, $A=e\cup B$ and thus $A\cap B' = (e\cup B)\cap B'
 =B$. This implies
\begin{equation}
    \label{equal12}
    a' = a\vee b'\quad \hbox{and}\quad b = a\wedge b'.
\end{equation}
Using the characterization in \eqref{equal12} and absorption, we have
\[
\rho(a',b') = \rho(a\vee b',b') = \rho(a,b') = \rho(a,a\wedge b') = \rho(a,b),
\]
which completes the proof.
\end{proof}

\begin{lemma}
\label{twostepequiv2}
Let $\CF\colon\sL\times\sL\to \sI$ be a monotone, distributive lattice form. 
Then, the map  
$\gamma\colon\sB^\updownarrow(\sL)\to \sI$   defined in \eqref{twostepequiv3} is an injective meet semilattice homomorphism.
\end{lemma}

\begin{proof}
We start with showing that $\gamma$ preserves the meet operation.
By Proposition~\ref{prop:LL},
both $\CCF$ and $\CF$ induce meet semilattice homomorphisms $\sL\times\sL\to\sI$.
Then,
\[
\begin{aligned}
\gamma\bigl((A\smin B) \cap (C\smin D)\bigr) &= \gamma\bigl((A\cap C)\smin (B\cup D) \bigr) 
= \CF(a\wedge c,b\vee d)\\
&= \CF(a,b)\wedge \CF(c,d) 
= \gamma(A\smin B) \wedge\gamma(C\smin D).
\end{aligned}
\]
By Proposition~\ref{lem:extraprop12}(ii), the function $\gamma$
satisfies $\gamma(\varnothing) = \CF(0,1)=0$ and $\gamma(\sfS(\sL)) = \CF(1,0)=1$, the neutral elements in the range $\gamma(\sB^\updownarrow(\sL))$. 
Moreover,  Proposition~\ref{lem:extraprop12}(iv) implies
 $\gamma(A\smin B) = \CF(a,b) = 0$ if and only if $a\leq b$ if and only if $A\smin B = \CCF(a,b) = \varnothing$.
Thus, $\gamma^{-1}(0) = \varnothing$.

It remains to show that $\gamma$ is injective.
Suppose $a,b,a',b'\in\sL$ such that 
\[
\CF(a,b) = \gamma(A\smin B) = \gamma(A'\smin B') = \CF(a',b') \quad \hbox{for} \quad A\smin B \neq A'\smin B'.
\]
Since $\gamma^{-1}(0) = \varnothing$, it follows that $A\smin B \neq \varnothing \neq A'\smin B'$ and $\CF(a,b) = \CF(a',b') \neq 0$. Let $C=A\smin B, D=A'\smin B' \in \sB^\updownarrow(\sL)$, then $\gamma(C) =\gamma(D) \not = 0$.
Recall that $\sB^\updownarrow(\sL) \subset \sB(\sL)$ and thus 
\[
(C\cup D) \smin (C\cap D) = (C\smin D) \cup (D\smin C)\neq \varnothing,
\]
since $C \neq D$.
Therefore, either $C\smin  D\neq\varnothing$ or $D\smin C \neq\varnothing$,
and we assume without loss
of generality the former holds.
From the description of the Priestley topology in Section~\ref{sec:booldual} every clopen subset of $\sfS(\sL)$ 
is a finite union of elements of $\sB^\updownarrow(\sL)$. Therefore there exist sets $\{E_i\in\sB^\updownarrow(\sL)~|~i=1,\ldots,n\}$
such that $C\smin D=\bigcup_i E_i$.
This implies that for 
$j\in\{1,\ldots,n\}$
\[
\varnothing \neq E_j \subset C\quad\text{and}\quad E_j\cap D = \varnothing.
\]
Observe that, since $\gamma$ is a semilattice homomorphism,
\[
\gamma(E_j)=\gamma(E_j\cap C)=\gamma(E_j)\cap\gamma(C)=\gamma(E_j)\cap\gamma(D)=\gamma(E_j\cap D)=\gamma(\varnothing)=0,
\]
which  is a contradiction since $\gamma^{-1}(0) = \varnothing$.
\end{proof}

\begin{corollary}
\label{charequivCF}
Suppose $\sI,\sI'$ are meet semilattices,  $\CF,\CF'\colon\sL\times\sL\to \sI,\sI'$ are monotone, distributive lattice  
forms.
Let $\gamma,\gamma'\colon \sB^\updownarrow(\sL)\rightarrowtail \sI,\sI'$  be the meet injective semilattice homomorphisms given by Theorem \ref{mainthm12}.
Then, $\CF'=g\circ\CF$ where 
\begin{equation}
\label{transition}
g=\gamma'\circ\gamma^{-1}\colon \sI_\CF \to \sI'_{\CF'},
\end{equation}
is an isomorophism.
\end{corollary}

There exits only one monotone, distributive lattice form up to isomorphisms which yields the equivalence class of monotone, distributive lattice forms and leads to the following definition.

\begin{definition}\label{CF}
Let $\sL$ be a bounded, distributive lattice, and let $\gamma \colon\sB^\updownarrow(\sL)\to\sI$ be a injective meet semilattice homomorphism.
The {\em Conley form on $\sL$ via $\gamma$} is 
\[
\CF := \gamma\circ\CCF \colon \sL\times\sL\to \gamma(\CCF(\sL\times\sL)) = \sI_\CF.
\]
Often it is the meet semilattice $\sI$ that is important, and the specific map $\gamma$ is implicitly defined from Theorem~\ref{mainthm12}, in which case we refer to {\em a representation of the Conley form in $\sI$.} 
If there is no ambiguity about the semilattice $\sI$ or homomorphism $\gamma$ we simply write $a-b:= \CF(a,b)$ to denote the Conley form for ease of notation.
\end{definition}

\begin{remark}
\label{rem:c0Boolean}
Observe that if $\sL$  is embedded in a Boolean algebra $\sE$, then 
there is a natural representation of the Conley form 
in $\sE$ itself with $\CF^b\colon\sL\times\sL\to\sE$ given by
 \[
 \CF^b(a,b)=a\smin b
 \]
as in Example~\ref{ex:X1}. The Conley form $\CF^b$ also implies a natural decomposition of
elements in $\sL$ which have a finite join-representation of the form $a=\bigvee_{a'\le a\atop a'\in \sJ(\sL)} a'$.
For such elements $a=\bigvee_{a'\le a\atop a'\in \sJ(\sL)} (a'\smin \pred a')$
\end{remark}

The decomposition given in Remark \ref{rem:c0Boolean} can be extended to lattices embedded into another lattice where the Boolean structure is replaced by a lattice form.
Let $\sL$ and $\sK$ be bounded distributive lattices with  $\sL\subset \sK$ and let
$\rho\colon \sL\times \sL \to \sK$ be a lattice form with the following additivity property
\begin{enumerate}
    \item[(Additivity)] $\rho(a,b)\vee b = a$, for all $b\le a$.
\end{enumerate}
This yields the following extension of the decomposition statement in Remark \ref{rem:c0Boolean}.
\begin{proposition}
\label{decompinlat}
If $a=\bigvee_{a'\le a\atop a'\in \sJ(\sL)} a'$ is finite join-representation, then
\begin{equation}
    \label{decompinlat2}
a= \bigvee_{a'\le a\atop a'\in \sJ(\sL)}\rho(a',\pred a') = \bigvee_{a'\le a\atop a'\in \sJ(\sL)}
\gamma(a'-\pred a'),
\end{equation}
where $a-b$ is the Conley form on $\sL$ in a semilattice $\sI$ and $\gamma\colon \sI \to \sK$ is given in Lemma \ref{twostepequiv1} and Corollary \ref{charequivCF}.
\end{proposition}

\begin{proof}
Let $a''$ be a maximal element in $\{a'\in \sJ(\sL)~|~a'\le a\}$, then $\pred a'' \le 
\bigvee \{a'\in \sJ(\sL)~|~a'\le a, a'\not = a''\}$. The additivity property of $\rho$ and induction on $a'$ give
\[
\begin{aligned}
a &= a'' \vee \bigvee \{a'\in \sJ(\sL)~|~a'\le a, a'\not = a''\}\\
&= 
\rho(a'',\pred a'') \vee \bigvee \{a'\in \sJ(\sL)~|~a'\le a, a'\not = a''\}
=  \bigvee_{a'\le a\atop a'\in \sJ(\sL)}\rho(a',\pred a').
\end{aligned}
\]
The latter statement in \eqref{decompinlat2} follows from Lemma \ref{twostepequiv1} and Corollary \ref{charequivCF}.
\end{proof}

\begin{remark}
\label{embedding}
If $\CF$ is a Conley form on $\sL$ represented in a semilattice $\sI$, then $\sL$ is naturally embedded in the convexity semilattice $\sI_\CF$, it has a natural dual lattice in $\sI_\CF$ and the notion of `complement' or `dual' is well-defined. 
The embedding of $\sL$ into $\sI$ is given by $a\mapsto \CF(a,0)$ and the dual of $a$ is defined as $a^*:= \CF(1,a)$. The dual lattice is given as $\sL^*=\{a^*~|~a\in \sL\}$.
As a consequence, $\sL$ and $\sL^*$ may be regarded as lattices in the same `universe' $\sI$.
Note that from distributivity we have that
\[
\CF(a,b) = \CF(a\wedge 1, 0\vee b) = \CF(a,0) \wedge \CF(1,b) = a\wedge b^*,
\]
which proves that every Conley form is can be characterized this way. For a homomorphism $h\colon \sK\to \sL$ there exists an induced dual anti-homomorphism $h^*\colon \sK^* \to \sL^*$ given by
$h^*(a^*) = h(a)^*$ and
\begin{equation}
\begin{diagram}
\label{commdual}
\node{\sK}\arrow{s,l}{h} \arrow{e,l}{^*}\node{\sK^*}\arrow{s,r}{h^*}\\
\node{\sL} \arrow{e,l}{^*}\node{\sL^*}
\end{diagram}
\end{equation}
\end{remark}

\section{Examples of Conley forms}
\label{exoflf}
In this paper we are interested in representations of the Conley form in the context of lattices of attractors and repellers, and we now provide some examples.

\begin{example}
\label{ex:X3}
In the context of invertible dynamical systems, 
attractors, repellers, and invariant sets all have lattice structures induced by intersection and union in the Boolean algebra $\sSet(X)$.
Indeed, $\sInvset(\varphi)$ is a complete (atomic) Boolean
subalgebra of $\sSet(X)$, and it contains all attractors and repellers, cf.\ App.\ \ref{sec:dynamics}.
Therefore, the Booleanizations of these lattices are isomorphic to a 
subalgebra of $\sInvset(\varphi)$ by Proposition~\ref{prop:booleanize2}.
In particular, in light of Remark~\ref{rem:c0Boolean}
\begin{equation}\label{eqn:rho}
\begin{aligned}
\CF^b \colon \sInvset(\varphi)\times \sInvset(\varphi) &\to \sInvset(\varphi) \\
(S,S') &\mapsto \CF^b(S,S') = S\smin S'.
\end{aligned}
\end{equation}
By Proposition~\ref{prop:restrictionMP} the restrictions of $\CF^b$ to 
$\sAtt(\varphi)$ and $\sRep(\varphi)$ are representations of the Conley forms 
of these lattices in $\sInvset(\varphi)$. 
\end{example}

\begin{example}
\label{ex:X4}
Let $\varphi\colon\T\times X\to X$ be an invertible dynamical system on a compact metric space.
Let $S\subset X$ be a compact invariant set and define the unstable set: $ W^u(S) := \setof{x\in X\mid \alpha(x)\subset S}$, cf.\ App.\ \ref{sec:topdyn}. For compact invariant sets $S,S'$ we have 
\[
W^u(S\cap A') = W^u(S)\cap W^u(S'),
\]
cf.\ Lemma \ref{lem:stlatthom} and Remark \ref{stunsth}.
By Example~\ref{ex:X3}, $\CF^b\colon \sAtt(\varphi)\times \sAtt(\varphi) \to \IS(\varphi)$ given in \eqref{eqn:rho} is a representation of the Conley form on $\sAtt(\varphi)$ in $\IS(\varphi)$.
To obtain an explict formula for $\CF^b$ in terms of $W^u$,
observe that if $A\in \sAtt(\varphi)$, then $W^u(A) = A$, and furthermore, by \cite[Theorem 3.19]{KMV-1a} $A^c = W^u(A^*)$ where $A^*$ is the dual repeller of $A$.
Therefore,
\[
\begin{aligned}
\CF^b(A,A')  &= A\smin A' = A \cap A'^c\\ &= W^u(A) \cap W^u\left( A'^*\right)\\ &= W^u( A\cap A'^*).
\end{aligned}
\]
Clearly, $A\cap A'^*\in \sMorse(\varphi)=\{A\cap R~|~A\in\sAtt(\varphi), R\in\sRep(\varphi)\}$.
Since $\sMorse(\varphi)$ is a subsemilattice of $\sInvset(\varphi)$ and $W^u\colon \sMorse(\varphi) \to \IS(\varphi)$ is injective, cf.\  Lemma \ref{injstunst},
\begin{equation}
\label{CFinvert}
\CF_\sAtt(A,A') = A-A' := A\cap A'^*
\end{equation}
is another (isomorphic)  representation of the Conley form of $\sAtt(\varphi)$ in $\IS(\varphi)$, cf.\ Theorem \ref{mainthm12} and Corollary \ref{charequivCF}. 
Since the dual operator $^*\colon\sAtt(\varphi)\to\sRep(\varphi)$  is an anti-isomorphism, c.f. \cite[Proposition 4.7]{KMV-1a}, 
\[
\sMorse(\varphi)=\CF_\sAtt(\sAtt(\varphi)\times\sAtt(\varphi)).
\]
\end{example}

\begin{example}\label{CFcomb}
Consider a binary relation $\cF\subset \cX\times \cX$ on a finite set $\cX$, see App.\ \ref{sec:combdyn}. 
Theorem~\ref{CFforref} establishes that the lattice form $\CF_\sAtt(\cA,\cA') :=\cA\cap \cA'^*$ is a representation of the Conley form on $\sAtt(\cF)$ in $\sInvset(\cF)$. 
By \cite[Diagram (5)]{KMV-1b} the dual operator $^*\colon\sAtt(\cF)\to\sRep(\cF)$  is an anti-isomorphism,
therefore, as in Example~\ref{ex:X4} 
\[
\CF_\sAtt(\sAtt(\cF)\times\sAtt(\cF))= \sMorse(\cF):=\{\cA\cap\cR~|~\cA\in\sAtt(\cF), \cR\in\sRep(\cF)\},
\]
where sets of the form $\cM=\cA\cap\cR$ are called \emph{Morse sets}.
\end{example}

\begin{lemma}
\label{invMorse}
Morse sets are invariant.
\end{lemma}

\begin{proof}
Suppose $\xi\in\cA\cap\cR$.
Since $\cA\in \sAtt(\cF)$, we have $\cF(\cA)=\cA$.
Therefore, there exists $\eta\in\cA$ such that $\xi\in\cF(\eta)$.
Similarly, $\cR\in \sRep(\cF)$ and hence $\cF^{-1}(\cR)=\cR$.
Thus, $\eta\in\cF^{-1}(\xi)\subset\cF^{-1}(\cR)=\cR$, and hence $\eta\in\cA\cap\cR$ and $\xi\in\cF(\cA\cap\cR)$. 
Therefore $\cA\cap\cR\subset\cF(\cA\cap\cR)$. 
The same argument applied to $\cF^{-1}$ gives $\cA\cap\cR\subset\cF^{-1}(\cA\cap\cR)$, which implies that $\cA\cap\cR\in \sInvset(\cF)$  by \cite[Proposition~3.4]{KMV-0}.
\end{proof}

\begin{lemma}
\label{charitoinv12}
Let $\cU\in \IS^+(\cF)$ and $\cV\in \IS^-(\cF)$. Then,
 $\Inv(\cU \cap \cV) = \bomega(\cU)\cap \balpha(\cV)$.
\end{lemma}

\begin{proof}
Since  $\bomega(\cU) \cap \balpha(\cV)\subset \cU\cap \cV$, Lemma \ref{invMorse} implies that 
$ \bomega(\cU) \cap \balpha(\cV)\subset\Inv(\cU\cap \cV)$.
Let $\cS\subset \cU\cap \cV$ be an invariant set. 
Since $\Inv(\cU)\subset\bomega(\cU)\subset \cU$ when $\cU$ is forward invariant,  $\cS\subset\Inv(\cU)\subset\bomega(\cU)$.
 Similarly, $\cS \subset \balpha(\cV)$ and therefore $\cS \subset \bomega(\cU) \cap \balpha(\cV)$.
\end{proof}

\begin{theorem}
\label{CFforref}
The lattice form
\begin{equation}
    \label{cfonrel}
   \CF_\sAtt(\cA,\cA') :=\cA\cap \cA'^*
\end{equation} 
is a representation of the Conley form in $\IS(\cF)$.
\end{theorem}

\begin{proof} 
For attractors $\cA, \cA'$ there exist $\cU,\cU'\in \IS^+(\cF)$ such that $\cA = \bomega(\cU)\subset \cU$ and $\cA' = \bomega(\cU')\subset \cU'$.
For the dual repellers $\cA^*,\cA'^*$ we have $\cA^* = \balpha(\cU^c)\subset \cU^c$ and $\cA'* = \balpha(\cU'^c)\subset \cU'^c$. In particular we can choose $\cU=\cA$ and $\cU' = \cA'$.
Observe that $\cA\cap \cA^* = \varnothing$. Indeed,
\[
\cA\cap\cA^* = \bomega(\cU) \cap \balpha(\cU^c) \subset \cU\cap \cU^c=\varnothing.
\]
Let $\cA,\cB,\cC,\cD\in\sAtt(\cF)$. Absorption is established by
\[
\begin{aligned}
\CF_\sAtt(\cA\vee \cB,\cA) & = (\cA\cup \cB) \cap A^* 
= (\cA\cap \cA^*) \cup (\cB\cap A^*) 
=\CF_\sAtt(\cB,\cA),
\end{aligned}
\]
and similarly 
\[
\begin{aligned}
\CF_\sAtt(\cA,\cA\wedge \cB) &= \cA\cap (\cA\wedge \cB)^* 
= \cA\cap (\cA^*\cup \cB^*) 
= (\cA\cap \cA^*) \cup (\cA\cap \cB^*) 
= \CF_\sAtt(\cA,\cB).
\end{aligned}
\]
Since $\cA\cap \cC\in \IS^+(\cF)$ and $\cB^*\cap \cD^*\in \IS^-(\cF)$ we have, using Lemma \ref{charitoinv12},
\[
\begin{aligned}
\CF_\sAtt(\cA,\cB) \wedge \CF_\sAtt(\cC,\cD) &= 
\Inv\bigl((\cA\cap \cB^*) \cap (\cC\cap \cD^*)\bigr)= \Inv\bigl((\cA\cap \cC) \cap (\cB^*\cap \cD^*) \bigr)\\
&= \bomega(\cA\cap\cC) \cap \balpha(\cB^*\cap \cD^*) = (\cA\wedge\cC)
\cap (\cB^*\wedge \cD^*)\\
&= (\cA\wedge\cC)
\cap (\cB\cup \cD)^*
= \CF_\sAtt(\cA\wedge \cC, \cB\cup\cD),
\end{aligned}
\]
which proves distributivity.
It remains to show that the lattice form is monotone.
Assume $\cA,\cA'\in\sAtt(\cF)$ satisfy $\CF_\sAtt(\cA,\cA') = \cA\cap \cA'^* = \varnothing$.
Observe that 
$\cA'^*\in \IS^-(\cF)$ and $\cU':=(\cA'^*)^c\in \IS^+(\cF)$. Since $\balpha(\cU'^c) = \cA'^*$
we have that $\bomega(\cU')= \cA'$.
Then,
\[
\varnothing = \cA\cap \cA'^* = \cA \smin (\cA'^*)^c,
\]
which implies that $\cA\subset (\cA'^*)^c$
and therefore $\cA = \bomega(\cA) \subset \bomega((\cA'^*)^c)= \cA'$
which establishes monotonicity and
 completes the proof.
\end{proof}

\section{Maps between Conley forms}
\label{mapslf}
We now discuss the effect of a lattice homomorphism on lattice forms and the Conley form in particular.
Theorems~\ref{mainthm12} and~\ref{mapCF} (below) imply that the Conley form behaves as a Boolean homomorphism under a homomorphism between lattices.
This confirms  that the Conley form is a generalization of the set difference operator for bounded, distributive lattices. 

\begin{theorem}
\label{mapCF}
Let $\sL$ and $\sK$ be bounded, distributive lattices and let
 $h\colon \sK\to \sL$ be a lattice homomorphism. 
For every representation of Conley forms on $\sL$ and $\sK$ if $a-b=a'-b'$, then
\[ 
h(a) - h(b) = h(a') -h(b').
\]
\end{theorem}

\begin{proof}
We use Diagram \eqref{commBool} for the Booleanization of $h$.
By construction of the Conley form on $\sK$ we have $a-b=a'-b'$ if and only if $A\smin B = A'\smin B'$.
Similarly,  for the Conley form on $\sL$ we have $h(a)-h(b)=h(a') -h(b')$ if and only if 
$j(h(a))\smin j(h(b))=j(h(a')) \smin j(h(b'))$. By Diagram \eqref{commBool}  the latter is equivalent to
\[
\sB(h)(A) \smin \sB(h)(B) = \sB(h)(A') \smin \sB(h)(B').
\]
Since $\sB(h)$ is Boolean it holds that $\sB(h)(A)\smin \sB(h)(B) = \sB(h)(A\smin B)$ which completes the proof.
\end{proof}

\begin{remark} The canonical Conley forms on $\sK$ and $\sL$ are
represented in Boolean algebras $\sB(\sK)$ and $\sB(\sL)$. The key 
idea in the above proof can be expressed as the fact that the Boolean
map $\sB(h)$ commutes with the canonical Conley forms, i.e.
\[
\CCF\circ (\sB(h)\times\sB(h))=\sB(h)\circ\CCF.
\]
\end{remark}

\begin{corollary}
\label{cor:existenceTheta}
Under the hypotheses of Theorem~\ref{mapCF}, $h$ induces a map $\theta\colon \sI_{\CF} \to \sJ_{\CF}$ given by
\[
\theta(a-b) := h(a) - h(b),
\]
and $\theta$ is a meet semilattice homomorphism preserving both neutral elements as expressed in the commutative diagram
\begin{equation}
\begin{diagram} 
\label{diag:AR60}
\node{\sK\times\sK} \arrow{e,l,A}{} \arrow{s,l}{h\times h}\node{\sI_{\CF}}\arrow{se,l}{\theta} \arrow{s,r}{\theta}\arrow{e,l,V}{ }\node{\sI}\\
\node{\sL\times \sL} \arrow{e,l,A}{ }  \node{\sJ_{\CF}}\arrow{e,l,V}{ }\node{\sJ} 
\end{diagram}
\end{equation}
\end{corollary}

\begin{proof}
For the homomorphism property we argue as follows. 
Using the distributivity of  the Conley forms on  $\sK$ and $\sL$ we have
\[
\begin{aligned}
\theta\bigl((a-b)\wedge (c-d)\bigr) &=\theta\bigl((a\wedge c) - (b\vee d)\bigr)
= h(a\wedge c)- h(b\vee d)\\
&=\bigl( h(a)\wedge h(c)\bigr) - \bigl(h(b)\vee h(d)\bigr)
=\bigl( h(a)-h(b)\bigr) \wedge \bigl(h(c)-h(d)\bigr)\\
&=\theta(a-b)\wedge \theta(c-d).
\end{aligned}
\]

For the neutral elements we have
\[
\theta(0-1) = h(0) - h(1) = 0-1,
\]
and similarly $\theta(1-0) = 1-0$ which shows that $\theta$ preserves the neutral elements in $\sI_{\CF}$ and
$\sJ_{\CF}$.
\end{proof}

\begin{remark}
\label{extendrho}
In Theorem \ref{mapCF}  
we can relax the Conley form on $\sL$ by a lattice form $\rho$ and 
the map $\theta$ is still well-defined
since only the absorption property is.
As for Corollary \ref{cor:existenceTheta} we still obtain a semilattice homomorphism $\theta$
if the Conley form on $\sL$ is relaxed to a distributive lattice form. 
\end{remark}

By Corollary \ref{cor:existenceTheta} and Remark \ref{extendrho} we can define the pullback of a lattice form. 
Let $h\colon \sK\to \sL$ be a lattice homomorphism and $\rho$ be a lattice form on $\sL$. 
Then,
\[
(h^\bullet\rho)(a,b) := \rho\bigl(h(a), h(b)\bigr),
\]
defines a lattice form on $\sK$.

\begin{corollary}
\label{pullback}
Let $h\colon \sK\to \sL$ be a lattice isomorphism, and let 
$\CF$ be a representation of the Conley form on $\sL$ in $\sI$, then $h^{\bullet}\CF$ is a representation of the Conley form on $\sK$ in $\sI$.
\end{corollary}

\begin{proof}
Distributivity follows from the proof of Corollary \ref{cor:existenceTheta}.
By definition $(h^\bullet\CF)(a,b) = h(a)-h(b)$.
By Theorem~\ref{mainthm12} to check that $h^\bullet\CF$ is a representation of the Conley form on $\sK$ we need to show  monotonicity. 
Consider $(h^\bullet)(a,b) = (h^\bullet)(0,1)$ which is equivalent to $h(a)-h(b) = h(0) - h(1) = 0-1$ in $\sI$. 
This implies that $h(a) \le h(b)$ and thus $a\le b$ since $h$ is an isomorphism.
\end{proof}

\begin{remark}
If $\sK\subset \sL$, then $h$ may be regarded as a lattice embedding in which case $h^\bullet\CF$ is the restriction of $\CF$ to $\sK$, cf.\ Proposition~\ref{prop:restrictionMP}.
\end{remark}

When $h\colon \sK\to\sL$ is an anti-homomorphism,
we define a pullback of a lattice form by
\begin{equation}\label{dualpull}
(h^\bullet\rho)(a,b):=\rho(h(b),h(a)),
\end{equation}
which is justified by the following proposition.

\begin{proposition}\label{prop:anti}
Let $h\colon \sK\to \sL$ be a lattice anti-isomorphism
 and let 
 $\CF$ be a representation of the Conley form on $\sL$ in $\sI$, then $h^{\bullet}\CF$ 
 is a representation of
the Conley form on $\sK$ in $\sI$.
\end{proposition}

\begin{proof}
To show that $h^\bullet\CF$ is a Conley form we verify absorption, distributivity, and monotonicity.
Consider
\[
\begin{aligned}
(h^\bullet\CF)(a\vee b,a)&=\CF(h(a),h(a\vee b))
=\CF(h(a),h(a)\wedge h(b))\\
&=\CF(h(a),h(b))
=(h^\bullet\CF)(b,a),
\end{aligned}
\]
and
\[
\begin{aligned}
(h^\bullet\CF)(a,a\wedge b)&=\CF(h(a\wedge b),h(a))
=\CF(h(a)\vee h(b),h(a))\\
&=\CF(h(b),h(a))
=(h^\bullet\CF)(a,b)
\end{aligned}
\]
which establishes absorption and
\[
\begin{aligned}
(h^\bullet\CF)(a\wedge c,b\vee d)&=\CF(h(b\vee d),h(a\wedge c))
=\CF(h(b)\wedge h(d),h(a)\vee h(c))\\
&=\CF(h(b),h(a))\wedge\CF(h(d),h(c))
=(h^\bullet\CF)(a,b)\wedge(h^\bullet\CF)(c,d)
\end{aligned}
\]
establishes distributivity. As for monotonicity we argue as follows.
Suppose $(h^\bullet\CF)(a,b)=h^\bullet\CF(0,1)$,
then $h(b)-h(a) = h(1)-h(0) = 0-1$.
Therefore, $h(b)\le h(a)$ which implies $a\le b$
since lattice anti-isomorphisms are order-reversing.
\end{proof}

\begin{example}
\label{ex:X4b}
Let $\varphi\colon\T^+\times X\to X$ be a dynamical system that is not necessarily   invertible.
The arguments in Example~\ref{ex:X4} make use of the fact that $\IS(\varphi)$ is a subalgebra of $\sSet(X)$.
For noninvertible dynamical systems, the meet lattice operation is not intersection, and hence  $\IS(\varphi)$ is not generally a sublattice of $\sSet(X)$.
Therefore, we need an alternative representation of a Conley form.
By Lemma~\ref{injstunst}
\[
\begin{aligned}
W^s\colon \sMorse(\varphi) &\to \IS^\pm(\varphi) \\
S & \mapsto W^s(S):= \setof{x\in X\mid \omega(x) \subset S}
\end{aligned}
\]
is an injective semilattice homomorphism.
Since $\sInvset^\pm(\varphi)$ is a Boolean algebra, following the same arguments as in Example~\ref{ex:X4}, using $W^s$ and Lemma \ref{lem:stlatthom} instead,  we obtain a representation of Conley form on $\sRep(\varphi)$ represented in $\IS(\varphi)$ as
\[
\CF_\sRep(R,R')=R-R' := R\cap R'^*
\]
with range $\sMorse(\varphi)=\CF_\sRep(\sRep(\varphi)\times\sRep(\varphi))$.

Since the dual operator $^*\colon\sAtt(\varphi)\to\sRep(\varphi)$  is an anti-isomorphism, c.f. \cite[Proposition 4.7]{KMV-1a}, Proposition~\ref{prop:anti} and Equation~\eqref{dualpull} imply that the pullback
\[
(h^\bullet \CF_\sRep) (A,A') = \CF_\sRep(A'^*,A^*) = A'^* \cap (A^*)^* = A\cap A'^*
\]
gives the following representation of the Conley form on $\sAtt(\varphi)$
in $\IS(\varphi)$ 
\[
\begin{aligned}
\CF_\sAtt\colon\sAtt(\varphi)\times\sAtt(\varphi)&\to\sMorse(\varphi)\\ 
(A,A') &\mapsto \CF_\sAtt(A,A') = A-A' :=A\cap A'^*,
\end{aligned}
\]
and $\sMorse(\varphi)=\CF_\sAtt(\sAtt(\varphi)\times\sAtt(\varphi))$.
\end{example}

\begin{remark}
In the remainder of the paper we will adopt
the notation $\CF_\sAtt(A,A')=A-A'$ and $\CF_\sRep(R,R')=R-R'$
indicated by the distinguished Conley forms $\CF_\sAtt$ and $\CF_\sRep$.
\end{remark}

\section{Conley forms and convexity semilattices for dynamical systems}
\label{sec:attmd}
We refine Corollary~\ref{cor:existenceTheta} in the context of various forms of dynamics.

\subsection{Combinatorial systems.}
\label{discreteCF}

Define the meet semilattice of \emph{Morse tiles}  to be
\[
\sMTile(\cF) := \CF^b(\sInvset^+(\cF)\times \sInvset^+(\cF))
\]
with $\CF^b(\cU,\cV) = \cU\smin \cV$. Then
Diagram~\eqref{diag:AR60} yields
\begin{equation}\label{diag:AR222c}
\begin{diagram}
\node{\sInvset^+(\cF)\times \sInvset^+(\cF)} \arrow{e,l,A}{\CF^b } \arrow{s,l,A}{  \omega\times \omega}\node{\sMTile(\cF)}\arrow{se,l}{\theta} \arrow{s,r,A}{ \theta}\arrow{e,l,V}{\subset}\node{\sSet(\cX)}\\
\node{\sAtt(\cF)\times \sAtt(\cF)} \arrow{e,l,A}{\CF}  \node{\sMorse(\cF)} \arrow{e,l,V}{\subset} \node{\sInvset(\cF)}.
\end{diagram}
\end{equation}
 The semilattice homomorphism $\theta\colon \sMTile(\cF) \twoheadrightarrow \sMorse(\cF)$  is defined by $\theta(\cU\smin \cU') = \bomega(\cU) - \bomega(\cU')=\cA-\cA'$ where $\cA=\bomega(\cU)$ and $\cA'=\bomega(\cU')$.
Since we have an explicit characterization of attractors via $\bomega$, we 
can  further characterize $\theta$. 
\begin{lemma}
\label{charitoinv}
 $\theta(\cU\smin\cU') = \Inv(\cU\smin \cU')$.
\end{lemma}

\begin{proof}
By Lemma \ref{charitoinv12}, $\cA-\cA' = \Inv(\cU\smin \cU')$.
\end{proof}

Lemma \ref{charitoinv} in combination with Diagram~\eqref{diag:AR222c} gives the following commutative diagram
\begin{equation}
\begin{diagram}
\node{\sInvset^+(\cF)\times \sInvset^+(\cF)} \arrow{e,l,A}{\CF^b } \arrow{s,l,A}{  \omega\times \omega}\node{\sMTile(\cF)} \arrow{s,r,A}{ \Small \Inv}\arrow{e,l,V}{\subset}\node{\sSet(\cX)}\arrow{s,l,A}{\Small \Inv}\\
\node{\sAtt(\cF)\times \sAtt(\cF)} \arrow{e,l,A}{\CF}  \node{\sMorse(\cF)} \arrow{e,l,V}{\subset} \node{\sInvset(\cF)}
\end{diagram}
\end{equation}

\subsection{Dynamical systems}
\label{dynsys12}
Example~\ref{ex:X4b} establishes a nontrivial representation of the Conley form on $\sAtt(\varphi)$ into $\sMorse(\varphi)$. In this setting, Diagram~\eqref{diag:AR60} applied to the lattice of closed attracting blocks, $\sABlockC(\varphi)$, yields

\begin{equation}\label{diag:AlphaOmega}
\begin{diagram}
\node{\sABlockC(\varphi)\times \sABlockC(\varphi)} \arrow{e,l,A}{\CF^b } \arrow{s,l,A}{  \omega\times \omega}\node{\sMTile(\varphi)} \arrow{s,l,A}{ \theta}\arrow{se,l}{\theta}\arrow{e,l,V}{\subset}\node{\sSet(X)} \\
\node{\sAtt(\varphi)\times \sAtt(\varphi)} \arrow{e,l,A}{\CF_\sAtt}   \node{\sMorse(\varphi)} \arrow{e,l,V}{\subset} \node{\sInvset(\varphi)}
\end{diagram}
\end{equation}
where $\CF^b(U,U') = U\smin U'$ is a Conley form in $\sSet(X)$ by Rmk.\ \ref{rem:c0Boolean}. The the range is
\[
\sMTile(\varphi):= \CF^b \left(\sABlockC(\varphi)\times \sABlockC(\varphi) \right)
\] 
which is called the meet semilattice of \emph{Morse tiles}. 
Recall that a set $U\subset X$ is an \emph{isolating neighborhood} if
$\Inv(\cl U) \subset \Int U$ and the associated \emph{isolated invariant set} is $S=\Inv(\cl U)$. The set of isolating neighborhoods $\sINbhd(\varphi)$ is a subsemilattice of $\sSet(X)$ and the set of isolated invariant sets $\sIsol(\varphi)$ is a subsemilattice of 
$\IS(\varphi)$.

The meet semilattice homomorphism $\theta\colon \sMTile(\varphi) \twoheadrightarrow \sMorse(\varphi)$ can be explicitly characterized.

\begin{lemma}
\label{tiles}
$A- A' =\theta(U\smin U') = \Inv(U\smin U') = \Inv\bigl(\cl(U\smin U')\bigr) \subset \Int(U\smin U')$ for all $U,U'\in \sABlockC(\varphi)$ and
$\sMTile(\varphi) \subset \sINbhd(\varphi)$ is a subsemilattice. In particular, Morse sets are isolated invariant sets.
\end{lemma}

\begin{proof}
Let $S\subset U\smin \Int U' = U\cap \cl(U'^c)$ be an invariant set. Then $S\subset U$, and thus $A\cup S\subset U$. Since $A=\Inv(U)$ it follows that $S\subset A$. Similarly, $S\subset \cl(U'^c)$  and thus $A^*\cup S\subset \cl(U'^c)$. Since $A^* = \Inv^+\bigl(\cl(U'^c)\bigr)$ it follows that $S\subset A'^*$. Consequently, 
$A-A' =\Inv(U\smin\Int U')$.
Since $\cl(U\smin U') \subset U\cap \cl(U'^c) = U\smin \Int U'$ it follows that
\[
\Inv\bigl(\cl(U\smin U') \bigr)\subset \Inv(U\smin\Int U') =A-A' \subset \Int(U\cap U'^c)
= \Int(U\smin U'),
\]
which proves that 
$U\smin U'$ is an isolating neighborhood. Because
$A-A'\subset U\smin U'\subset \cl(U\smin U')$ it follows that $A-A'= \Inv(U\smin U') = \Inv\bigl(\cl(U\smin U')$. The fact that
$\sMTile(\varphi)$ is a subsemilattice of $\sSet(X)$ implies it is a subsemilattice of $\sINbhd(\varphi)$.
\end{proof}

Refining Diagram~\eqref{diag:AlphaOmega} based on Lemma~\ref{tiles} gives

\begin{equation}\label{diag:AlphaOmega2}
\begin{diagram}
\node{\sABlockC(\varphi)\times \sABlockC(\varphi)} \arrow{e,l,A}{\CF^b } \arrow{s,l,A}{  \omega\times \omega}\node{\sMTile(\varphi)} \arrow{s,l,A}{ \Small \Inv}\arrow{e,l,V}{\subset}\node{\sINbhd(\varphi)}\arrow{s,l,A}{ \Small \Inv} \\
\node{\sAtt(\varphi)\times \sAtt(\varphi)} \arrow{e,l,A}{ \CF_\sAtt}   \node{\sMorse(\varphi)} \arrow{e,l,V}{\subset} \node{\sIsol(\varphi).} 
\end{diagram}
\end{equation}
The fact that $\Inv\colon \sINbhd(\varphi) \twoheadrightarrow \sIsol(\varphi)$ is
a semilattice homomorphism follows from \cite[Lemma 2.7]{KMV-1a}.

\begin{remark}
In the above commutative diagram we could also have chosen to use the lattice of attracting neighborhoods in place of attracting blocks.
In this case, the image of the Conley form is a larger subsemilattice of the isolating neighborhoods.
In the next section we present Morse tiles in the setting of regular closed sets which arise naturally in computations, \cite{KMV-1b,KastiKV}.
\end{remark}

\subsection{Regular closed sets} 
\label{sec:regularClosed}

As indicated in \cite{KMV-1b}, for computational purposes it is useful to define Conley forms in the setting of regular closed sets $\scrR(X)$, cf.\ App.\ \ref{sec:closedRegular}.
The set of closed regular sets that are attracting blocks is denoted by $\sABlockR(\varphi)$.
The goal of this section is to prove that the following is a commutative diagram of lattice homomorphisms
\begin{equation}\label{diag:AR111}
\begin{diagram}
\node{\sABlockC(\varphi)\times \sABlockC(\varphi)} \arrow{e,l,A}{\CF^b } \arrow{s,l,A}{  ^{\rc\rc}\times ^{\rc\rc} }\node{\sMTile(\varphi)}\arrow{se,r}{\theta_{\rc\rc}} \arrow{s,l,A}{\theta_{\rc\rc}}\arrow{e,l,V}{\subset}\node{\sIBlock(\varphi)}\\
\node{\sABlockR(\varphi)\times \sABlockR(\varphi)} \arrow{e,l,A}{\CF^b } \arrow{s,l,A}{  \omega\times \omega}\node{\sMTile_\scrR(\varphi)} \arrow{s,l,A}{\Small \Inv} \arrow{e,l,V}{\subset}\node{\sIBlockR(\varphi)}\arrow{s,r,A}{\Small \Inv}\\
\node{\sAtt(\varphi)\times \sAtt(\varphi)} \arrow{e,l,A}{\CF_\sAtt }   \node{\sMorse(\varphi)} \arrow{e,l,V}{\subset}\node{\sIsol(\varphi)}
\end{diagram}
\end{equation}
where $U^{\rc\rc} := \cl \Int U$.

\begin{remark}
\label{rem:ClosedRegularSuffices}
Observe that since the top and bottom rows are as in \eqref{diag:AlphaOmega2} and the vertical maps are surjective, there is no information lost by working with regular closed  sets.
\end{remark}

\begin{lemma}
\label{lem:closedBToRegularB}
If $U\in \sABlockC(\varphi)$, then $U^{\rc\rc}\in \sABlockR(\varphi)$.
\end{lemma}

\begin{proof}
By assumption $\varphi(t,U^{\rc\rc}) \subset \varphi(t,U) \subset \Int U = \Int U^{\rc\rc}$ for all positive $t\in \T$, where the latter
follows the fact that $\Int U = U^{\bot\bot} = U^{\bot\bot\bot\bot} = \Int \cl \Int U= \Int U^{\rc\rc}$ and $U^\bot = (\cl U)^c$.
\end{proof}

The map $^{\rc\rc}\colon \sABlockC(\varphi) \to\sABlockR(\varphi)$ is a lattice homomorphism by Lemma \ref{regclhom} and $\omega\colon \sABlockR(\varphi) \to \sAtt(\varphi)$ is a lattice homomorphism by  \cite[Theorem 3.15]{KMV-1b}.
As a consequence, we obtain the following three commutative diagrams of lattice homomorphisms.
First,
\begin{equation}
    \label{eq:CtoR}
\begin{diagram}
\node{\sABlockC(\varphi)}\arrow[2]{e,l,A}{^{\rc\rc}}\arrow{se,r,A}{\omega}\node{}\node{\sABlockR(\varphi)}\arrow{sw,r,A}{\omega}\\
\node{}\node{\sAtt(\varphi)}
\end{diagram}
\end{equation}
where the surjectivity of $^{\rc\rc}$ follows from $\sABlockR(\varphi)\subset \sABlockC(\varphi)$.
Furthermore, by Diagram~\eqref{diag:AR60}
\begin{equation}\label{diag:AR111t}
\begin{diagram}
\node{\sABlockC(\varphi)\times \sABlockC(\varphi)} \arrow{e,l,A}{\CF^b} \arrow{s,l,A}{  ^{\rc\rc}\times ^{\rc\rc}}\node{\sMTile(\varphi)} \arrow{se,l}{{\theta_{\rc\rc}}}\arrow{s,r,A}{\theta_{\rc\rc}}\arrow{e,l,V}{\subset}\node{\sIBlock(\varphi)}\\
\node{\sABlockR(\varphi)\times \sABlockR(\varphi)} \arrow{e,l,A}{\CF^b } \node{\sMTile_\scrR(\varphi)}   \arrow{e,l,V}{\subset}\node{\sIBlockR(\varphi)}
\end{diagram}
\end{equation}
and
\begin{equation}\label{diag:AR111b}
\begin{diagram}
\node{\sABlockR(\varphi)\times \sABlockR(\varphi)} \arrow{e,l,A}{\CF^b } \arrow{s,l,A}{  \omega\times \omega}\node{\sMTile_\scrR(\varphi)}\arrow{se,l}{\theta_\scrR} \arrow{s,r,A}{\theta_\scrR} \arrow{e,l,V}{\subset}\node{\sIBlockR(\varphi)}\\
\node{\sAtt(\varphi)\times \sAtt(\varphi)} \arrow{e,l,A}{\CF_\sAtt }   \node{\sMorse(\varphi)} \arrow{e,l,V}{\subset}\node{\sIsol(\varphi).}
\end{diagram}
\end{equation}

The Conley form on $\sABlockC(\varphi)$ is given by $\CF^b(U,U') = U\smin U'$ for $U,U'\in \sABlockC(\varphi)$. 
The Conley form on $\sABlockR(\varphi)$ is given by
\begin{equation}
\label{eq:ConleyFormRegClosed}
\CF^b(N,N') = N\wedge N'^\rc = \cl (N \smin N')\quad \hbox{for $N,N'\in \sABlockR(\varphi),$}
\end{equation}
where the latter follows from Lemma \ref{lem:difference}.

Consider the  homomorphism $^{\rc\rc}\colon \sABlockC(\varphi) \to \sABlockR(\varphi)$.
Via Corollary \ref{cor:existenceTheta} and Eqn.\ \eqref{defn:wedgeRegular} the induced meet semilattice homomorphism $\theta\colon\sMTile(\varphi) \to \sMTileR(\varphi)$ is given by 
\[
\theta_{\rc\rc}(U\smin U') = U^{\rc\rc} - U'^{\rc\rc} := U^{\rc\rc}\wedge U'^{\rc\rc\rc} = \cl (N\smin N')
\]
where $N := U^{\rc\rc}$.
From   Lemma \ref{tiles}, Corollary \ref{cor:existenceTheta} and Lemma \ref{lem:difference} we derive that
\[
\theta_\scrR(N-N') = \omega(N)-\omega(N') = \Inv(N\smin N') = \Inv(\cl(N\smin N'))
= \Inv(N-N'),
\]
which proves the following lemma.
\begin{lemma}
\label{charregcl12}
$A-A' = \theta_\scrR(N - N') = \Inv(N-N')$.
\end{lemma}
Lemma \ref{charregcl12} together with \eqref{diag:AR111b} yields the bottom half of \eqref{diag:AR111}. The duality between regular closed attracting and repelling blocks is given by the following lemma which explains regular closed Morse tiles as regular intersections of attracting blocks and repelling blocks and
 characterizes the duality in Diagram \eqref{commdual} in this case.
\begin{lemma}
$\sABlockR(\varphi) \xleftrightarrow{^\rc} \sRBlockR(\varphi)$.
\end{lemma}

\begin{proof}
From  \cite[Lemma 3.17]{KMV-1a} we derive that if $U\in \sRBlockR(X,\varphi)$, then $U^\rc$ satisfies
\[
\varphi(-t,U^\rc) \subset\Int U^c = U^c=\Int \cl U^c = \Int U^\rc,\quad t<0.
\]
The latter follows from the fact that $U^c$ is a regular open set.
From \cite[ Lemma 3.17]{KMV-1a} we also derive that if $U\in \sABlockR(X,\varphi)$, then
\[
\varphi(-t,U^\rc) \subset U^c = \Int U^\rc, \quad t>0,
\]
 which proves that $U^\rc \in \sRBlockR(X,\varphi)$.
\end{proof}

\section{Representations of lattices}
\label{sec:MRandMD}

We have shown that attractors in a dynamical system have the structure of a bounded, distributive lattice, which codifies algebraically the global structure of the dynamical system. From a dynamics point of view, this global structure has been alternatively described in terms of a poset of distinguished invariant sets, the order of which encodes the global structure. From an algebraic point of view, a bounded, distributive lattice is dually equivalent to a poset via Priestley duality as described in Section~\ref{sec:booldual}. Hence, the order on the Priestley space is dynamically defined, and the central issue is the representation of the Priestley space as a poset of invariant sets.

In the previous sections we have identified dynamically distinguished invariant sets, namely the Morse sets, which can be characterized as the image of the specific Conley form $\CF_\sAtt$ on the lattice of attractors represented in the invariant sets. In particular, $\sMorse(\varphi)\cong\sB^{\updownarrow}(\sAtt)$. 
The Conley form is designed to provide a representation of the semilattice structure of $\sB^{\updownarrow}(\sL)$ in a more meaningful semilattice $\sI$.
However, since the Booleanization functor $\sB=\sOcl\circ \sF\circ \sfS$  forgets the order on $\sfS(\sL),$ a representation of $\sfS(\sL)$ as a poset in $\sI$ does not immediately follow. In this section, we show that in the case of finite lattices, the Priestley space can indeed be represented as a poset consisting of elements in $\sI$, but the issue is more subtle in the infinite case.

\subsection{Spectral representations}
\label{subsec:repr1}
Let $\sL$ be a finite distributive lattice. Then the convexity semilattice $\sB^\updownarrow(\sL)$ is the lattice of all convex sets $\sConvex(\sfS(\sL))$ in the spectrum $\sfS(\sL)$. 
The lattice $\sConvex(\sfS(\sL))$ is a (complete) atomic lattice, which is not distributive in general, cf.\ \cite{BirkhoffBen}. The anti-chain of atoms in $\sConvex(\sfS(\sL))$ contains exactly the sets $\{I\}$ where
$I\in \sfS(\sL)$ ignoring the order structure.
Let $\CF\colon \sL\times \sL\to \sI$ be a Conley form on $\sL$, then 
$
\sB^\updownarrow(\sL) \cong \CF(\sL\times\sL) = \sI_\CF,
$
and $\CF$ determines the injective semilattice homomorphism $\gamma\colon\sB^\updownarrow(\sL)\to \sI_\CF$
given by 
\[
\gamma(A\smin B) = \CF(j^{-1}(A),j^{-1}(B))\quad \hbox{for $A,B\in \sB^\down(\sL)$}.
\]
Writing $\{I\}$ as $\{I\} = \down I \smin (\down I \smin \{I\})$ gives
\[
\gamma(\{I\}) = \CF(j^{-1}\bigl(\down I\bigr),j^{-1}\bigl(\down I \smin \{I\}\bigr)).
\]
Since the join irreducible elements of $\sB^\down(\sL)$ are exactly those of the form $\down I$ for $I\in \sfS(\sL)$, the join irreducible elements of $\sL$ are exactly $j^{-1}(\down I)$ for $I\in \sfS(\sL)$, since $j$ is an isomorphism, cf.\ Sect.\ \ref{sec:booldual}.
Consequently, $\{\CF(a,\pred{a})~|~a\in\sJ(\sL)\}=\{\gamma(\{I\})~|~I\in\sfS(\sL)\}$ is a representation of $\sfS(\sL)$ in $\sI$.
This motivates the following definition.
\begin{definition}
Let $\sL$ be a finite, distributive lattice and 
let $\CF\colon \sL \times \sL \to \sI$ be a representation of the Conley form in a meet semilattice $\sI$
and  $\gamma\colon \sB^\updownarrow(\sL)\to \sI$ be the injective semilattice homomorphism given by $\CF$.
The \emph{spectral representation of $\sfS(\sL)$ in $\sI$} is defined to be the poset $\sM(\sL)$ where
\[
\sM(\sL) :=\setof{\gamma\bigl(\{I\}\bigr) \mid I\in \sfS(\sL)} = \setof{a-\pred a \mid a\in \sJ(\sL)}
\]
and
\begin{equation}
\label{mumap2}
\begin{aligned}
\gamma(\setof{I}) \le \gamma(\setof{I'})\quad &\Longleftrightarrow \quad I\subset I',\quad \hbox{ or equivalently}\\
a-\pred a \le a'-\pred a' \quad &\Longleftrightarrow \quad a\le a'.
\end{aligned}
\end{equation}
Since $\gamma$ is an isomorphism, 
\[
\begin{aligned}
\ji\colon \sfS(\sL) &\to \sM(\sL) \\
I &\mapsto \gamma\bigl(\{I\}\bigr) = a-\pred a \quad \hbox{for $a=j^{-1}(\down I).$}
\end{aligned}
\]
is an order isomorphism.
\end{definition}

\begin{lemma}\label{rem:pairdisjoint}
As elements in $\sI_\CF$ it holds that
$\gamma(\{I\} \not = 0$  for all $I\in \sfS(\sL)$ and 
 $\gamma(\{I\})\wedge\gamma(\{I'\})=0$ for all $I\not =  I'$.
 \end{lemma}

\begin{proof}
Since $\pred a < a$  Proposition \ref{lem:extraprop12}(iv) implies that $a-\pred a\not = 0$.
Note that elements of $\sM(\sL)$ are pairwise disjoint  since $\gamma(\{I\})\wedge\gamma(\{I'\})=\gamma(\{I\}\cap\{I'\})=0$ and $\gamma$ is an isomorphism $\sB^\updownarrow(\sL)\to\sI_\CF.$
\end{proof}

\begin{figure}[hbt]
\centering
{\includegraphics[width=13cm]{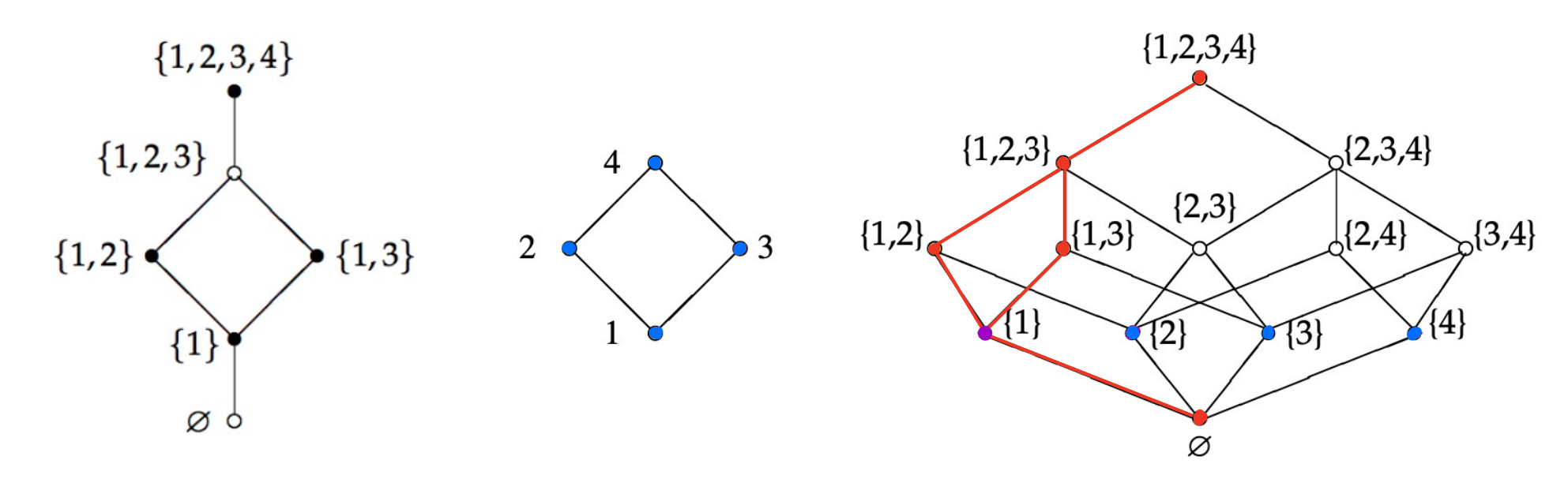}}
\caption{ A finite lattice $\sL=\sO(\sP)$ [left], the poset $\sP$ representing the spectrum [middle] and
the convexity semilattice [right]. The subset in red [right] gives an embedding of $\sL$ into the convexity semilattice, cf.\ Remark \ref{embedding}.} 
\label{1pointpic2} 
\end{figure}

The above construction 
implies that in the finite case, the (clopen) singleton, convex sets in $\sB^\updownarrow(\sL)$
are in one-to-one correspondence with $\sJ(\sL)$
and $\sfS(\sL)$ which is used to construct spectral representations.
In general, when $\sL$ is infinite, $\sB^\updownarrow(\sL)$ is only a subsemilattice of $\sConvex(\sfS(\sL))$, \cite{BirkhoffBen}. Indeed, there are infinite, bounded, distributive lattices that possess clopen, singleton, convex sets that are not associated to a join irreducible elements. 
In this case, representing the spectrum is more subtle and will not be addressed in this paper, cf.\ \cite{GKV}. 
For dynamics, the finite case is often sufficient as we are interested in Morse representations which are finite.

Since $\sM(\sL)$ is order-isomorphic to $\sfS(\sL)$, Birkhoff's Representation Theorem implies  that $\sO(\sM) \cong \sL$.
We denote this isomorphism by $\al\colon \sO(\sM) \to \sL$. 
Let $\CF$ denote the Conley form on $\sL$ represented in $\sI$. 
This gives the diagram:
\begin{equation}
\label{diag:AR551}
\begin{diagram}
\node{\sO(\sM)\times \sO(\sM)} \arrow{e,l,A}{\CF^\sigma } \arrow{s,lr}{\al\times \al}{\cong}\node{\sB^\updownarrow(\sO(\sM))} \arrow{s,lr}{\theta}{\cong}\\
\node{\sL \times \sL} \arrow{e,l,A}{\CF}\node{\sI_\CF} 
\end{diagram}
\end{equation}
where by commutativity the final isomorphism is given by 
\[
\theta(\alpha\smin \beta) = \al(\alpha)-\al(\beta) \hbox{ for all $\alpha,\beta\in\sO(\sM)$, }
\]
and consequently $\theta(\{M\}) = M$ for all $M\in \sM(\sL)$.

When we apply the spectral representation in the dynamical setting using $\CF_\sAtt$, we use the terminology of a {\em Morse representation} in place of spectral representation. For example, consider $\sL = \sAtt(\varphi)$ and the Conley form 
\[
\CF_\sAtt(A,A') = A-A' :=A\cap A'^*.
\]
For a finite sublattice $\sA\subset\sAtt(\varphi)$ the associated Morse representation is given by
\[
\sM(\sA) = \{A-A'~|~A\in \sJ(\sA)\}.
\]
In the combinatorial setting, the Morse representation for any sublattice $\sA$ of $\sAtt(\cF)$ is also
\[
\sM(\sA) = \{\cA-\cA'~|~\cA\in \sJ(\sA)\}.
\]

\subsection{Stable and unstable set representations}
\label{sec:stableandunstable}
Consider the maps
 $W^u\colon \sMorse(\varphi) \to \IS(\varphi)$ and
$W^s\colon \sMorse(\varphi) \to \IS^\pm(\varphi)$  where the latter is an injective semilattice homomorphism, cf.\ Lemma~\ref{injstunst}.
These maps induce a lattice form $\rho^u$ and a Conley form $\CF^s$ given by
\[
\rho^u := W^u \circ \CF_{\sAtt}\quad \text{and}\quad
\CF^s := W^s \circ \CF_{\sRep}.
\]
From the fundamental theorem of attractor-repeller pairs \cite[Theorem 3.19]{KMV-1a}
we have that 
$A\smin \pred A \subset \rho^u(A,\pred A)$, and therefore $\rho^u$ satisfies the additivity property in Proposition \ref{decompinlat}.
Let $\al\colon \sO(\sM)\to \sAtt(\varphi)$ be the injective lattice homomorphism with range $\sA$ from Birkhoff's Representation Theorem.
Then, for $\alpha \in \sO(\sM)$, we have $\alpha = \bigcup_{\alpha'\subset \alpha\atop \alpha'\in \sJ(\sO(\sM))} \alpha'$ and consequently, $A = \bigcup_{A'\subset A\atop A'\in \sJ(\sA)} A'$ where $A'= \al(\alpha')$.
Proposition \ref{decompinlat} 
implies
\begin{equation}
    \label{decompo1}
\al(\alpha) =  \bigcup_{A'\subset A\atop A'\in \sJ(\sA)} \rho^u(A',\pred A') = \bigcup_{M\in \alpha} W^u(M),
\end{equation}
which provides an explicit expression for $\al(\alpha)$. Similarly, since also $\CF^s$ satisfies the additivity property in Proposition \ref{decompinlat}, a representation for repellers can be obtained via the homomorphism
$\al^*\colon \sU(\sM) \to \sRep(\varphi)$ given by
\begin{equation}
    \label{decompo2}
    \al^*(\beta) =  \bigcup_{M\in \beta} W^s(M).
\end{equation}
From Remark \ref{embedding} we have the following commutative diagram
\begin{equation}
\label{diag:AR567}
\begin{diagram}
\node{\sO(\sM)} \arrow{e,l,<>}{^c} \arrow{s,lr}{\al}{\cong}\node{\sU(\sM))} \arrow{s,lr}{\cong}{\al^*}\\
\node{\sA} \arrow{e,l,<>}{^*}\node{\sA^*} 
\end{diagram}
\end{equation}
where we have identified $\CF_{\sAtt}(\omega(X), A) = \omega(X)\cap A^*$ with $A^*$, since $\omega(X)\cap A^*\mapsto W^s\bigl(\omega(X)\cap A^*\bigr) = A^*$ is injective, cf.\ Lemma \ref{injstunst}. This yields the correspondence 
\begin{equation}
    \label{decompo3}
    A=\nu(\alpha)=\bigcup_{M\in \alpha} W^u(M) \;\;\Longleftrightarrow\;\;
    A^*=\al(\alpha)^* = \al^*(\alpha^c)= \bigcup_{M\in \alpha^c} W^s(M).
\end{equation}

\begin{remark}
We leave it to the reader to verify that the homomorphisms $\theta$ and $\theta^*$ induced by $\al$ and $\al^*$ respectively coincide, i.e. $\theta(\alpha\smin \beta) = \theta^*(\alpha\smin \beta)$.
\end{remark}

We can use the above decompositions to obtain a decomposition in terms of connecting orbits. We now show  that the decompositions in \eqref{decompo1} and \eqref{decompo2} can be utilized to relate the partial order on $\sM(\sA)$ to the dynamics of $\varphi$.
Theorem \ref{thm:ji} below provides a dynamical description of this order, which serves as an extension of an attractor-repeller pair and can be used
as a dynamical definition of a Morse representation.

\begin{theorem}\label{thm:ji}
Let $ \sM(\sA)$ be a Morse representation subordinate to a finite sublattice $\sA\subset \sAtt(\varphi)$. 
 Then, the sets $M\in \sM(\sA)$ are compact, nonempty, pairwise disjoint, invariant sets in $X$, and for every $x\in X$ there exists $M\in\sM$ such that $\omega(x) \subset M.$ Moreover, for every complete orbit $\gamma_x$ with $x\not \in \bigcup_{M\in\sM}M$ 
there exist $M,M'\in\sM$  with  $M<M'$ such that
 $\omega(x) \subset M$ and  $\alphaOg \subset M'.$
\end{theorem}

\begin{proof}
By definition $M=A-\pred A = A\cap \pred A^*\in \sM(\sA)\subset \sInvset(\varphi)$. 
Since $M$ is the  intersection of an attractor and repeller, it is compact and isolated, cf.\  \cite{KMV-1a}
and Lemma \ref{tiles}.
Finally, by Lemma~\ref{rem:pairdisjoint} Morse sets $M$ are nonempty.
Furthermore, by Lemma \ref{rem:pairdisjoint},  $M\wedge M' = \varnothing$ 
 for all
 $M\not = M'$.
This implies that $M\cap M' = \varnothing$ for all $M\not = M'$. 
Indeed, the intersection $M\cap M'$ is compact and forward invariant and thus $M\wedge M' = \Inv(M\cap M') = \omega(M\cap M')$ is nonempty unless $M\cap M'=\varnothing$. 

The decompositions in   \eqref{decompo1} and \eqref{decompo2} imply that
\begin{equation}
  \label{decompo2a}  
X=\bigcup_{M\in\sM}W^s(M)\quad \hbox{and}\quad \omega(X) = \bigcup_{M\in\sM}W^u(M),
\end{equation}
so that for each $x\in X$ there exists $M$ such that $x\in W^s(M)$ and thus $\omega(x)\subset M$.  
Let $\gamma_x$ be a complete orbit with $x\in X\smin \bigr(\bigcup_{M\in \sM} M\bigl)$. 
Then, by the decompositions in \eqref{decompo2a} we have that $x\in W^s(M) \cap W^u(M')$. 
By definition $\omega(x) \subset M$ and $\alphaOg\subset M'$.
 It remains to show that $M<M'$.
 Suppose $M>M'$ or $M\Vert M'$ and write the singleton convex set $\{M\}$ in $\sO(\sM)$ as
  $\{M\} = \alpha \smin \beta$ with $\alpha=\downarrow M\in\sO(\sM)$ and $\beta = (\uparrow M)^c\in\sO(\sM)$, and likewise
   $\{M'\} = \alpha' \smin \beta'$ with $\alpha'=\downarrow M'$ and $\beta' = (\uparrow M')^c$.
 Then,
\[
\begin{aligned}
W^s(M) \cap W^u(M') &= \CF^s\bigl(\al(\alpha),\al(\beta)\bigr) \cap \rho^u\bigl(\al(\alpha'),\al(\beta')\bigr)\\
&=W^s\bigl(\al(\alpha)\cap \al(\beta)^*\bigr) \cap W^u\bigl(\al(\alpha') \cap \al(\beta')^*\bigr)\\
&\subset W^s(\al(\beta)^*) \cap W^u(\al(\alpha')) = \al(\beta)^* \cap \al(\alpha') = \al(\alpha') - \al(\beta)
\end{aligned}
\]
By the mapping property of the Conley form we have
\[
\al(\alpha')-\al(\beta) = \theta(\alpha'\smin \beta) =\varnothing,
\]
since $\alpha'\smin \beta = \down M' \cap \up M =\varnothing$
by the assumptions on $M$ and $M'$, which proves that $M<M'$.
\end{proof}

\subsection{Reconstruction of  attractor lattices}
\label{sec:classMD}

Theorem \ref{thm:ji} establishes dynamical properties of a Morse representation. The next result shows that
the characterization in Theorem \ref{thm:ji} can be used as a dynamical definition of Morse representations.

\begin{theorem}
\label{thm:mainMDthm}
Let $\sM$ be a finite poset of nonempty, pairwise disjoint, compact, invariant sets in $X$. Then $\sM$ 
is a Morse representation subordinate to a finite sublattice $\sA(\sM)\subset \sAtt(\varphi)$ if and only if 
for every $x\in X$ there exists $M\in\sM$ such that 
$\omega(x)\subset M$, and for each
complete orbit $\gamma_x$ with 
$x\not \in \bigcup_{M\in \sM} M$
  there exists $M<M'$ such that $\omega(x)\subset M$ and  $\alphaOg \subset M'.$
The associated 
lattice 
$\sA(\sM)$ is the image of the injective lattice homomorphism $\al\colon\sO(\sM) \to \sInvset(\varphi)$ given by 
\begin{equation}\label{eq:numap}
\alpha \mapsto \al(\alpha) = 
\bigcup_{M\in \alpha} W^u(M) \subset \sAtt(\varphi),
\end{equation}
and $\sM = \sM\bigl(\sA(\sM)\bigr)$.
\end{theorem}

The ``only if" direction is Theorem~\ref{thm:ji}, so the proof of Theorem~\ref{thm:mainMDthm} is divided into two lemmas in which we assume the second set of conditions stated in the theorem.

\begin{lemma}
\label{lem:MDrec1}
Let $M_0\in \sM$ be a minimal element. 
Then, $M_0$ is an attractor.
\end{lemma}

\begin{proof}
The set $X'=\omega(X)$ is a compact metric space, and the restriction $\varphi'=\varphi|_{X'}$ is a surjective  dynamical system  on $X'$.
Due to the invariance of both $X'$ and the sets $M\in \sM$ we have that for every $x\in X'\smin (\bigcup_{M\in \sM} M)$,
  there exists $M<M'$ such that $\omega(x) \subset M$ and
$\alphaOg \subset M'$.
By assumption
we can choose a compact neighborhood $N\supset M_0$ with $N\cap M=\varnothing$
for all $M\not = M_0$. For $x \in N\smin M_0$ the assumptions imply that $\alphaOg \subset M \not =M_0$ for all backward orbits $\gamma_x^-$.
Consequently, there are no backward orbits $\gamma_x^-\colon \T^-\to N$ for all $x\in N\smin M_0$. By \cite[Lemma 3.11]{KMV-1a}
the set $M_0$ is an attractor for $\varphi'$. By \cite[Proposition 3.7]{KMV-1a}, $M_0$ is also an attractor for $\varphi$.
\end{proof}

\begin{lemma}
\label{lem:MDrec2}
Let $M_0 \in \sM$ be a minimal element. Then
\[
R = \bigcup_{M\not = M_0} W^s(M) 
\]
is the   repeller dual to $A=M_0$.
\end{lemma}

\begin{proof}
By Proposition 3.16 in \cite{KMV-1a}, since $M_0$ is an attractor, cf.\ Lemma \ref{lem:MDrec1},  the dual attractor of $M_0$ is characterized by $M_0^* = \{x\in X~|~\omega(x) \cap M_0 = \varnothing\}$.
Suppose $x\in R$, then $\omega(x) \subset M$ for some $M\not = M_0$, and therefore $R\subset M_0^*$.
Conversely, if $x\in M_0^*$, then $\omega(x) \subset M$ with $M \not = M_0$,
which implies $M_0^*\subset R$, and thus $M_0^* = R$.
\end{proof}

\begin{proof}[Proof of Theorem \ref{thm:mainMDthm}]
Since the sets $W^s(M)\in \IS^\pm(\varphi)$ for $M\in \sM$ are mutually disjoint sets in $\sI = \IS^\pm(\varphi)$,
the map
\begin{equation}
\label{eqn:injrep12}
\begin{aligned}
\al^*\colon\sU(\sM) &\rightarrowtail \IS^\pm(\varphi)\\
\beta &\mapsto \bigcup_{M\in \beta} W^s(M) 
\end{aligned}
\end{equation}
defines a injective lattice homomorphism and the range is denoted by $\sA^*(\sM)$, cf.\ Lemma \ref{injstunst}.

Lemmas \ref{lem:MDrec1} and \ref{lem:MDrec2} show that 
$R=\bigcup_{M\not = M_0} W^s(M) = \al^*\bigl((\down M_0)^c\bigr)$
are repellers for all minimal elements in $M_0 \in \sM$.
Let $X'$ be the intersection of these repellers, which is again a repeller, and
let $\sM'$ be the poset obtained from $\sM$ by removing all minimal elements. Then, $\sM'$ satisfies the conditions of Theorem \ref{thm:mainMDthm} in $X'$.
Repeat the above lemmas in $X'$. By \cite[Proposition 3.28]{KMV-1a} repellers in $X'$ are repellers in $X$, and thus by exhausting the poset $\sM$ we establish that all elements of the form  $\al^*\bigl((\down M)^c\bigr)$, $M\in \sM$,
are repellers in $X$. 
Since the elements $\al^*\bigl((\down M)^c\bigr)$ are meet-irreducible, all elements in $\sA^*(\sM)$, except for $\al^*(\sM)$, are meets of meet-irreducible repellers. By \eqref{decompo3}
$\al^*(\sM) = X$, a repeller, which establishes $\al^*\colon \sU(\sM)\to \sRep(\varphi)$ as a injective lattice homomorphism. 
Consequently, $\alpha \mapsto \al^*(\alpha^c)^* \in \sAtt(\varphi)$ is a injective lattice homomorphism
$\al\colon \sO(\sM) \to \sAtt(\varphi)$ by Diagram~\eqref{diag:AR567}.

 Moreover, for $\alpha\smin \beta = \{M\}$ we have
\[
\begin{aligned}
\al(\alpha) -\al(\beta) &= \al^*(\alpha^c)^*-\al^*(\beta^c)^*\\
&=\al^*(\alpha^c)^*\cap \al^*(\beta^c) = \al^*(\beta^c) - \al^*(\alpha^c)\\
&= \theta^*(\beta^c\smin \alpha^c)  = \theta^*(\alpha\smin \beta) = M,
\end{aligned}
\]
which implies, by \eqref{decompo3}, that $\al$ is given by \eqref{eq:numap} completing
the proof.
\end{proof}

Given a finite attractor lattice $\sA\subset \sAtt(\varphi)$, then $\sM(\sA)$ satisfies the hypotheses of Theorem \ref{thm:mainMDthm}
and the associated attractor lattice $\sA(\sM)$ in \eqref{eq:numap} is isomorphic to $\sA$ due to Birkhoff's Representation Theorem and the attractors coincide,
which shows that $\sA(\sM(\sA))  = \sA$. 
We conclude
\begin{equation}
\label{eqn:firstid}
\sA\circ \sM = \id\quad\hbox{and}\quad \sM\circ \sA = \id.
\end{equation}

\begin{remark}
The above characterization and construction of Morse representations can also be implemented for finite binary relations $\cF\subset \cX\times\cX$. In \cite[Defn.\ 3.9]{KMV-0} a dynamical definition
of Morse representation is given. The results in Theorem \ref{thm:ji} and Theorem \ref{thm:mainMDthm} also hold in this setting.
\end{remark}

\section{Morse decompositions}
\label{subsec:repr2}
Let $\sP$ be a finite poset.
 A lattice homomorphism $\sO(\sP) \to \sL$ can be factored through its range $\sA\subset \sL$, i.e. 
$\sO(\sP) \twoheadrightarrow \sA \rightarrowtail \sL$.
This yields the factorization 
$\sfS(\sL) \twoheadrightarrow \sfS(\sA) \hookrightarrow \sfS(\sO(\sP))$.
Given a spectral representation $\sM(\sA)$ via a Conley form on $\sA$ and Birkhoff's Representation Theorem, we obtain
\[
\sfS(\sL) \twoheadrightarrow \sM(\sA) \hookrightarrow \sP.
\]
In the context of dynamics, we make the following definition.
\begin{definition}
\label{defn:MD1}
Let $\sP$ be a finite poset and $\sA$ be the image of a lattice homomorphism $\sO(\sP)\to\sAtt$. The order-embedding $\pi\colon\sM(\sA) \hookrightarrow \sP$ is called the \emph{Morse decomposition} dual to the lattice epimorphism
$\sO(\sP) \twoheadrightarrow \sA$.
\end{definition}

The term Morse decomposition was first defined in Conley  theory in  the setting of continuous time dynamical systems via labelings of collections of invariant sets by a poset whose order is consistent with the dynamics, cf.\  \cite{Conley}.
By reformulating this concept in terms of embeddings of posets we obtain
a formulation of Morse decomposition consistent with the algebraic theory developed in this paper.
Here we emphasize the algebraic nature of a Morse decomposition as an order-embedding from a Morse representation into a poset.
The importance of the role of the poset $\sP$ and the information it provides about the dynamical system becomes most apparent in computations
where the poset $\sP$ is the computable object, cf.\ Section~\ref{substr}.
Generally, we refer to a Morse decomposition without mentioning the dual lattice homomorphism.

\subsection{Tessellated Morse decompositions}
\label{tessel}
In this section we present a dynamically meaningful choice of poset $\sP$ in a Morse decomposition.
Let $\sN\subset \sABlockR(X,\varphi)$ be a finite sublattice  of regular closed attracting blocks, and
consider the  Conley form given in Section~\ref{sec:regularClosed} in the setting of regular closed attracting neighborhoods.
From~\ref{sec:closedRegular} and Diagram~\eqref{diag:AR111} 
we derive the commutative diagrams
\begin{equation}\label{diag:AR5155}
\begin{diagram}[1.9]
\node{\sN} \arrow{e,l,<>}{^\rc} \arrow{s,l,A}{ \omega}  \node{\sN^\rc}\arrow{s,r,A}{  \alpha} \\
\node{\sA} \arrow{e,l,<>}{^*}   \node{\sA^*}
\end{diagram}
\qquad\qquad
\begin{diagram}
\node{\sN \times \sN} \arrow{s,l,A}{\omega\times \omega}\arrow{e,l,A}{\CF^b }   \node{\sMTile_\scrR(\sN)} \arrow{s,r,A}{\Small  \Inv}\\
\node{\sA \times \sA} \arrow{e,l,A}{\CF_\sAtt}   \node{\sMorse(\sA)} 
\end{diagram}
\end{equation}
where $\sMTile_\scrR(\sN)$ denotes the image of $\CF^b$ restricted to $\sN\times\sN$ and $\sMorse(\sA)$ denotes the image of $\CF_\sAtt$ restricted to $\sA\times\sA$.
As in Section~\ref{subsec:repr1} we obtain a spectral representation in $\sMTile_\scrR(\sN)$ that is called a \emph{Morse tessellation} and denoted by
\[
\sT(\sN) = \bigl\{ T=N-\pred N ~|~ N\in \sJ(\sN)\bigr\}
\]
where $N-\pred N=  \cl(N \smin \pr N)$, and $N-\pred N\le N'-\pred N'$ if and only if $N\subset N'$. 
Hence $\sT = \sT(\sN)$ is a poset, and the map $\sJ(\sN) \to \sT(\sN)$ given by $N\mapsto N-\pred{N}$ is an order-isomorphism.
Moreover, 
the functoriality of Birkhoff's Representation Theorem yields the order-embedding $\pi\colon\sM(\sA) \hookrightarrow \sT(\sN)$ subordinate to the lattice surjection $\omega\colon \sN \twoheadrightarrow \sA$.

\begin{definition}
\label{defn:TMD}
Let $\sT(\sN)$ be a Morse tessellation of regular closed sets subordinate to the sublattice
$\sN \subset \sABlockR(\varphi)$. Then, the homomorphism $\pi\colon\sM(\sA)   \hookrightarrow \sT(\sN)$ is called a \emph{tessellated Morse decomposition}
subordinate to  $\omega\colon\sN\twoheadrightarrow \sA$.
\end{definition}
For a given  tessellated Morse decomposition
$\pi\colon \sM(\sA) \hookrightarrow \sT(\sN)$, the Morse tessellation $\sT(\sN)$ plays the role of the poset $\sP$ in the definition of Morse decomposition.

\begin{remark}\label{rmk:aboutp}
Observe that by Corollary~\ref{cor:existenceTheta} $\theta\colon \CF^b\bigl(\sO(\sT(\sN)) \times \sO(\sT(\sN))\bigr) \to \CF_\sAtt(\sA\times \sA)$ is a semilattice homomorphism.
Furthermore, by Lemma~\ref{tiles} $\theta=\Inv$.
The map $\pi\colon \sM(\sA) \hookrightarrow \sT(\sN)$ is a order-embedding. 
If we identify the elements of $\sT(\sN)$ with the singleton sets in $\CF^b\bigl(\sO(\sT(\sN)) \times \sO(\sT(\sN))\bigr)$, then 
\[
\Inv \circ \pi = \theta\circ \pi = \id_{\sM(\sA)},
\]
and $\theta=\Inv$ acts as a left-inverse for $\pi$.
Therefore, the Morse sets can be recovered as the maximal invariant sets within the Morse tiles.
Observe that since $\varnothing$ is in the range of $\theta$, we capture the possibility that the maximal invariant set in a Morse tile may be empty. 
\end{remark}

\begin{remark}
\label{differMT}
One can also define tessellated Morse decomposition via $\sANbhd(\varphi)$ or $\sANbhdR(\varphi)$.
\end{remark}

\subsection{Spans and combinatorial models}
\label{substr}

For a given dynamical system $\varphi$ in this paper we have constructed the following \emph{span} in the category of bounded distributive lattices
\[
\begin{diagram}
\node{\scrR(X)} \node{\sABlockR(\varphi)}\arrow{w,l,V}{\supset}\arrow{e,l,A}{\omega} \node{\sAtt(\varphi)}
\end{diagram}.
\]
Spans can be used to define equivalence classes of dynamical systems based on their gradient behavior. Two dynamical systems $(X,\varphi)$ and $(Y,\psi)$ are \emph{span equivalent} if 
there exist isomorphisms such that following diagram commutes
\[
\begin{diagram}
\node{\scrR(X)} \arrow{s,l,<>}{}\node{\sABlockR(\varphi)} \arrow{s,l,<>}{}\arrow{w,l,V}{\supset}\arrow{e,l,A}{\omega} \node{\sAtt(\varphi)} \arrow{s,l,<>}{}\\
\node{\scrR(Y)} \node{\sABlockR(\psi)}\arrow{w,l,V}{\supset}\arrow{e,l,A}{\omega} \node{\sAtt(\psi).}
\end{diagram}
\]
The analogue of a span in the category of finite distributive lattices is given by
\begin{equation}\label{finite99}
\begin{diagram}
\node{\sSet(\cX)} \node{\sO(\sP)}\arrow{w,l,V}{\iota}\arrow{e,l,A}{h} \node{\sO(\sQ),}
\end{diagram}
\end{equation}
where $\cX$ is a finite set and $\sP$ and $\sQ$ are finite posets from Birkhoff's Representation Theorem.
Spans in the category of finite distributive lattices can be equivalently described through finite binary relations $\cF\subset \cX\times \cX$.
To be more precise the extension of Birkhoff's Representation Theorem in \cite{KastiKV} yields the following
representation of a finite span in terms of a binary relation $\cF$, i.e.\
\eqref{finite99} can be equivalently described by
\begin{equation*}
\label{model2aabd}
\begin{diagram}
\node{\sSet(\cX)} 
\node{\IS^+(\cF)}\arrow{w,l,V}{\supset}\arrow{e,l,A}{\bomega} \node{\sAtt(\cF),}
\end{diagram}
\end{equation*}
where $\bomega$ is the omega limit set in the setting of binary relations, cf.\ Eqn.\ \eqref{eqn:bomega} and \cite{KMV-1b}.
We emphasize that
 the choice of $\cF$ is not unique, cf.\ \cite{KastiKV}.
The next step is to consider diagrams of the form
\begin{equation}
\label{model212}
\begin{diagram}
\node{\scrR(X)}\node{\sABlockR(\varphi)} \arrow{w,l,V}{\supset}\arrow{e,l,A}{\omega} \node{\sAtt(\varphi)} \\
\node{\sSet(\cX)}\arrow{n,l}{|\cdot |} \node{\IS^+(\cF)}\arrow{n,l}{|\cdot |}\arrow{w,l,V}{\supset}\arrow{e,l,A}{\bomega} \node{\sAtt(\cF),}\arrow{n,l}{c}
\end{diagram}
\end{equation}
where the second homomorphism is a restriction of the first.
We show that if the third homomorphism exists and the diagram commutes, then it is uniquely defined by $c=\omega(|\cdot|)$.

\begin{remark}
Typically in applications, $\cX$ is a labeling of
the atoms of
a subalgebra of regular closed sets, ie.~a grid cf.\ \cite{KMV-1b}, and the map $|\cdot|$ is the evaluation map
$$
|\cU|=\bigcup_{\xi\in\cU}|\xi|,
$$
which is injective. Also, in the definition of span one may consider sublattices of $\scrR(X)$ and $\sSet(\cX)$, which is useful in some applications.
\end{remark}

 We refer to the above diagram as a \emph{commutative combinatorial model} for $\varphi$, see \cite{KastiKV}. %
Recall from \cite{KMV-0,KMV-1a,KMV-1b} that a way to combinatorialize a dynamical system is to discretize both time and space. 
In this section we explain combinatorialization from an algebraic point of view.
In order to do so we introduce two hypotheses.
First, a finite binary relation   $\cF$ is called a \emph{weak outer approximation} if
\begin{enumerate}
    \item [(W)] $\varphi\bigl(t,|\xi|\bigr) \subset \Int |\Gamma^+(\xi)|$
for all $t>0$,
\end{enumerate}
where $\Gamma^+(\xi)$ denotes the forward image of $\xi$ under $\cF$.
The commutativity of the first square in \eqref{model212} is equivalent to (W) by Theorem~5.3 in \cite{KastiKV}.
In order to characterize commutativity of the second square in \eqref{model212}  we use an additional 
criterion for $\cF$
\begin{enumerate}
    \item [(L)] $\omega(|\xi|) \subset |\bomega(\xi)|$ for all $\xi \in \cX$.
\end{enumerate}

\begin{theorem}
\label{critWL}
Let $\cF\subset \cX\times \cX$ be a finite, binary relation. 
Diagram \eqref{model212} commutes if and only if $\cF$ satisfies {\em (W)} and {\em (L)}. 
In this case $c=\omega(|\cdot|)$.
\end{theorem}

\begin{proof}
If the second square in \eqref{model212} commutes, then $\omega(|\cU|) = c(\bomega(\cU))$ for every $\cU\in \IS^+(\cF)$. In particular, each $\cA\in\sAtt(\cF)$ satisfies $\cA=\bomega(\cA)$ and is an element of $\IS^+(\cF)$ so that
$\omega(|\cA|) = c(\cA)$, i.e. $c = \omega(|\cdot|)$.

Let $\cU = \Gamma^+(\xi)$ be the complete forward image of some $\xi\in \cX$. Then, 
\[
\omega(|\xi|) \subset \omega(|\cU|) = c(\bomega(\cU)) = \omega(|\bomega(\cU)|)\subset |\bomega(\cU)|
= |\bomega\bigl( \Gamma^+(\xi)\bigr)| = |\bomega(\xi)|,
\]
which establishes property (L).

Conversely, suppose (L) is satisfied.
For $\cU\in \IS^+(\cF)$ we have that $\bomega(\cU) \subset \cU$ and therefore
$
\omega(|\bomega(\cU)|) \subset \omega(|\cU|).
$
Moreover, since $\omega$, $\bomega$, and $|\cdot |$ are homomorphisms, 
\[
\omega(|\cU|) = \bigcup_{\xi\in \cU} \omega(|\xi)|)
\subset \bigcup_{\xi\in \cU} |\bomega(\xi)| = |\bomega(\cU)|,
\]
and thus $\omega(|\cU|) = \omega\bigl(\omega(|\cU|) \bigr) \subset \omega(|\bomega(\cU|).$ Combining both inclusions, $\omega(|\bomega(\cU)|)= \omega(|\cU|)$.
This establishes the commutativity of the second square in \eqref{model212} when $c=\omega(|\cdot|).$
\end{proof}

From this point on we assume that $|\cdot|$ is injective.
If we consider the diagram in \eqref{model212} by denoting the ranges of the bottom span we obtain
\begin{equation}
\label{modelfinal1}
\begin{diagram}
\node{\scrR_0}\node{\sN} \arrow{w,l,V}{\supset}\arrow{e,l,A}{\omega} \node{\sA} \\
\node{\sSet(\cX)}\arrow{n,l,<>}{|\cdot |} \node{\IS^+(\cF)}\arrow{n,r,<>}{|\cdot |}\arrow{w,l,V}{\supset}\arrow{e,l,A}{\bomega} \node{\sAtt(\cF),}\arrow{n,l,A}{\omega(|\cdot|)}
\end{diagram}
\end{equation}
where $\scrR_0$ is the algebra of grid elements, $\sN$ is a finite lattice of attracting blocks, and 
$\sA$ is a finite lattice of attractors.
We now invoke the various Conley forms to dualize the above diagrams which yields the following dual diagram
\begin{equation}
\label{modelfinal2}
\begin{diagram}
\node{|\cX|}\arrow{e,l,A}{}\node{\sT(\sN)}  \node{\sM(\sA)}\arrow{w,l,L}{\pi}\arrow{s,l,J}{} \\
\node{\cX}\arrow{e,l,A}{\supset}\arrow{n,l,<>}{|\cdot |} \node{\sSCC(\cF)}\arrow{n,l,<>}{|\cdot |} \node{\sRC(\cF),}\arrow{w,l,L}{}
\end{diagram}
\end{equation}
which provides a factorization of the tessellation $|\cX|\twoheadrightarrow \sT(\sN)$ and
the tessellated Morse decomposition $\sM(\sA) \hookrightarrow \sT(\sN)$. Together these define the co-span
\[
|\cX|\twoheadrightarrow \sT(\sN) \hookleftarrow \sM(\sA).
\]
The posets $\sSCC(\cF)$ and $\sRC(\cF)$ are the spectral representations of
$\IS^+(\cF)$ and $\sAtt(\cF)$ respectively, cf.\ Section~\ref{subsec:repr1} and
\cite{KMV-1b,KastiKV}.
The dual diagram shows that binary relations $\cF$ that satisfy Hypotheses (W) and (L) give rise to tessellated Morse decompositions. This fact has been used to computationally characterize and compare global dynamics
in various contexts, \cite{database,KMV-0,KMV-1b,KastiKV,Lyap,DSGRN0,DSGRN1,database2,DK,BK,BKM}.

Given a tessellation $\sT$ of $X$ consisting of regular closed sets labeled by $\cX$.
If we choose $\cF$ to be transitive and reflexive, then $\cF$ is a partial order on $\cX$ and induces a partial order on $\sT = |\cX|$.  By Theorem \ref{critWL} we can then formulate the following equivalent characterization of Morse tessellations in the spirit Theorem \ref{thm:mainMDthm} for Morse representations.

\begin{corollary}
\label{MTchar1}
A (finite) poset $(\sT,\le)$ consisting of a regular closed partition $\sT$ of $X$ is a Morse tessellation
subordinate to a (finite) sublattice of attracting blocks if and only if
the partial order is a weak outer approximation for $\varphi$.
\end{corollary}

\begin{proof}
Let $\cF$ be the partial order induced on $\cX$ by the poset $\sT$. Assume that $\cF$ is a weak outer approximation.
Since $\cF$ is a partial order, $\bomega(\xi) = \down \xi$ for every $\xi\in \cX$. Then (W) implies $\omega(|\xi|)\subset |\down\xi|=|\bomega(\xi)|$ so that (L) is satisfied.
The remainder follows from Theorem \ref{critWL}
and fact that finite sublattices of $\sABlockR(\varphi)$ yield partially ordered partitions of regular closed sets for which the down-sets give attracting blocks by construction.
\end{proof}

We refer to \cite{KastiKV} for a more detailed account on combinatorial models and applications.

\bibliographystyle{plain}
\bibliography{KMVc-biblist}

\appendix

\section{Dynamics}
\label{sec:dynamics}

In this section we recall definitions from dynamical systems theory, from both continuous and discrete time dynamical systems as well as the dynamics of finite relations. 
\subsection{Topological Dynamics}
\label{sec:topdyn}
In  this paper we use the following definition of a continuous dynamical system, cf.\ \cite{KMV-1a,KMV-1b}.
\begin{definition}
\label{defn:local-dyn}
Let $\T$ denote either $\Z$ or $\R$.
A \emph{dynamical system} is a continuous map $\varphi\colon \T^+ \times X \to X$ that satisfies
\begin{enumerate}
\item [(i)] $\varphi(0,x) = x$ for all $x\in X$, and
\item [(ii)] for all $s,t\in \T^+$ and for all $x\in X$ it holds that 
$\varphi(t,\varphi(s,x)) = \varphi(t+s,x)$.
\end{enumerate}
If $\varphi\colon \T \times X \to X$ satisfies (i) and (ii), then $\varphi$ is called an \emph{invertible} dynamical system.
\end{definition}
Recall that $A\subset X$ is an \emph{attractor} for $\varphi$ if there exists an open neighborhood $U$ of $A$ such that $\omega(U)=A$, and dually, 
$R\subset X$ is a \emph{repeller} for $\varphi$ if there exists an open neighborhood $U$ of $R$ such that $\alpha(U)=R$.
The bounded, distributive lattice of attractors and repellers is denoted by $\sAtt(\varphi)$ and $\sRep(\varphi)$, respectively.
The binary relations on $\sAtt(\varphi)$ are $A\wedge A':= \omega(A\cap A')$ and $A\vee A' := A\cup A'$, and on $\sRep(\varphi)$ the binary relations are intersection and union. In \cite{KMV-1a}, it is shown that there is a natural well-defined duality anti-isomorphism $^*\colon\sAtt(\varphi)\to\sRep(\varphi)$ via $A=\omega(U) \mapsto \alpha(U^c)=A^*$. The pair $(A,A^*)$ is called an {\em attractor-repeller pair}.

A set $U\subset X$ is an \emph{attracting neighborhood}  if $\omega(\cl U)\subset \Int U$  and  a \emph{repelling neighborhood}   if $\alpha(\cl U)\subset \Int U$. The collection of all attracting and repelling neighborhoods form bounded distributive lattices, $\sANbhd(\varphi)$ and $\sRNbhd(\varphi)$, respectively,
with binary operations intersection and union.
As shown in \cite{KMV-1a}, the map $\omega \colon \sANbhd(\varphi)\twoheadrightarrow \sAtt(\varphi)$ is a surjective lattice homomorphism.
Similarly, $\alpha\colon \sRNbhd(\varphi) \twoheadrightarrow \sRep(\varphi)$ is a surjective lattice homomorphism. 
A subset $U\subset X$ is an \emph{attracting block} for $\varphi$ if
\[
\varphi(t,\cl U) \subset \Int U\quad \forall t>0.
\]
The set of closed attracting blocks of $\varphi$ is denoted by $\sABlockC(\varphi)$.
By \cite[Lemma~3.3]{KMV-1a} and the fact that intersection and union of closed sets are closed, $\sABlockC(\varphi)$ is a bounded distributive lattice. 
Since the inclusion $\sABlockC(\varphi) \rightarrowtail \sSet(X)$ is a lattice homomorphism, we can define the dual lattice
\[
\sRBlockO(\varphi)  := \setof{U^c \mid U\in \sABlockC(\varphi)}.
\]
Furthermore, by \cite[Lemma~3.17]{KMV-1a}, we have that $\sRBlockO(\varphi)\rightarrowtail \sSet(X)$, and complement acts as an anti-lattice isomorphism between $\sRBlockO(\varphi)$ and $\sABlockC(\varphi)$.
From the  perspective of dynamics,  $V\in \sRBlockO(\varphi)$ if and only if $\varphi(t,\cl V) \subset \Int V$ for all  $t<0$.

\begin{lemma}
\label{lem:omegaABlock}
Let $\varphi\colon\T^+\times X \to X$ be a dynamical system on a compact metric space.
Then, $\sAtt(\varphi)= \omega(\sABlockC(\varphi))$ and $\sRep(\varphi) = \alpha(\sRBlockO(\varphi))$.
\end{lemma}

\begin{proof}
Let $A\in \sAtt(\varphi)$.
Because $\sAtt(\varphi):= \omega(\sANbhd(\varphi))$, there exist  $U \in \sANbhd(\varphi)$ such that $\omega(U) = A$.
By \cite[Prop.\ 3.5]{KMV-1a} there exists a trapping region $\widehat U\subset U$ such that $\omega(U) = \omega(\widehat U) = A$.
By \cite[Lemmas 6.5 and 7.7]{KMV-0} there exists a Lyapunov function $V\colon X\to [0,1]$ such that $V^{-1}(0) = A$, $V^{-1}(1) = A^*$, and $V(\varphi(t,x))< V(x)$ for all $t>0$ and $x\not \in A\cup A^*$, where $A^*$ is the dual repeller to $A$. 
Due to compactness we can choose $0<\epsilon\ll 1$ such that $N =\setof{x\in X \mid V(x)\le \epsilon}\subset \widehat U\subset U$ is a closed attracting block with $\omega(N) = \omega(\widehat U) = \omega(U) = A$.
Therefore, $A\in \omega(\sABlockC(\varphi))$.
The proof that $\sRep(\varphi) = \alpha(\sRBlockO(\varphi))$ is similar.
\end{proof}

The following result is a Corollary of \cite[Proposition 3.16]{KMV-1a}.

\begin{lemma}
\label{lem:alphaOmegaDiagram}
Let $\varphi\colon\T^+\times X \to X$ be a dynamical system on a compact metric space.
Then, the following diagram commutes:
\begin{equation}\label{diag:AR22}
\begin{diagram}
\node{\sABlockC(\varphi)} \arrow{e,lr}{^c}{\cong} \arrow{s,l,A}{ \omega}\node{\sRBlockO(\varphi)} \arrow{s,r,A}{ \alpha}\\
\node{\sAtt(\varphi)} \arrow{e,lr}{^*}{\cong}   \node{\sRep(\varphi)} 
\end{diagram}
\end{equation}
\end{lemma}
The upper homomorphism follows from the proof of Lemma~3.17 in \cite{KMV-1a}.

\begin{definition}
\label{standunst}
For a compact 
invariant set $S\subset X$    define the sets
\[
\begin{aligned}
W^s(S) &=\{x\in X~|~ \omega(x) \subset S\};\\
W^u(S) &= \{x\in X~|~\exists \gamma_x^- ~\ni~ \alphaOg \subset S\},
\end{aligned}
\]
which are called the \emph{stable} and \emph{unstable} sets of $S$ respectively.
\end{definition}

\begin{lemma}
\label{stunstinv}
The stable set $W^s(S)$ is forward-backward invariant, and the unstable set $W^u(S)$ is invariant.
\end{lemma}

\begin{proof}
Let $x\in W^s(S)$. 
Then $\omega(\varphi(t,x)) \subset \omega(x)\subset S$ for every $t\in \T$. 
Therefore, $\omega(\varphi(t,x)) \subset S$
for all $x\in W^s(S), t\in\T$,
which proves that $W^s(S)$ is both forward and backward invariant.

As for $W^u(S)$, we argue as follows. Let $x\in W^u(S)$. Then a complete orbit $\gamma_x$ exists. Therefore, every $y\in \gamma_x$ has a backward orbit $\gamma_y^-$, and $\alpha_{\rm o}(\gamma_y^-) = \alphaOg \subset S$,
which proves that $\gamma_x\subset W^u(S)$ for all $x\in W^u(S)$ and establishes the invariance of $W^u(S)$.
\end{proof}

\begin{lemma}
\label{lem:stlatthom}
Let $S,S'$ be compact invariant sets. Then, $W^s(S\wedge S') = W^s(S) \cap W^s(S')$.
\end{lemma}

\begin{proof}
We have $W^s(S\wedge S')\subset W^s(S) \cap W^s(S')$. Now 
let $x\in W^s(S) \cap W^s(S')$. Then $\omega(x)\subset S\cap S'$, and since $\omega(x)$ is invariant, $\omega(x) = \omega(\omega(x)) \subset \omega(S\cap S') = S\wedge S'$,
and thus $W^s(S) \cap W^s(S') \subset W^s(S\wedge S')$.
\end{proof}

\begin{remark}
\label{stunsth}
The same property with respect to union is not clear unless the invariant sets are attractors. 
The equivalent of Lemma \ref{lem:stlatthom} does not hold for $W^u(S\wedge S')$.
If $\varphi$ is an invertible system, i.e. a dynamical system
with time $t\in \Z$, or $t\in \R$, then 
then we can use the proof of Lemma \ref{lem:stlatthom} to show that
both $W^s$ and $W^u$ define lattice homomorphisms from the sublattice of compact invariant sets
to the invariant sets.
\end{remark}

\begin{lemma}
\label{lem:stlatthom0}
Let $A\in \sAtt(\varphi)$. Then,
the application  $W^s\colon \sAtt(\varphi) \to \IS^\pm(\varphi)$, defined by $A\mapsto W^s(A)$, is a lattice embedding.
\end{lemma}

\begin{proof}
By Theorem 3.19 in \cite{KMV-1a} we have that $W^s(A) = \{x\in X~|~\omega(x)
\subset A\} = (A^*)^c$. This implies
\[
\begin{aligned}
W^s(A) \cup W^s(A') & = (A^*)^c \cup (A'^*)^c
 = (A^*\cap A'^*)^c 
= \bigl((A\cup A')^*\bigr)^c
= W^s(A\cup A').
\end{aligned}
\]
Similarly,
\[
\begin{aligned}
W^s(A) \cap W^s(A') & = (A^*)^c \cap (A'^*)^c 
= (A^*\cup A'^*)^c 
= \bigl((A\wedge A')^*\bigr)^c
= W^s(A\wedge A').
\end{aligned}
\]
To prove that the homomorphism is injective we argue as follows. 
Suppose $W^s(A) = W^s(A')$, then equivalently $(A^*)^c = (A'^*)^c$. Since both $^*$ and $^c$ are involutions, we have that $A=A'$,
which completes the proof.
\end{proof}

The following lemma is an extension of \cite[Prop.\ 3.21]{KMV-1a}.
\begin{lemma}
\label{aux}
Let $A\subset X$ be an attractor and let $N$ be a compact set satisfying
$A\subset N \subset W^s(A)$. Then, $\omega(N) = A$.
\end{lemma}

\begin{proof}
By definition $(W^s(A))^c = A^*$ and thus $N\cap A^*=\varnothing$ by the assumptions on $N$.
Since compact metric spaces are normal, there exist  open sets 
separating $N$ and $A^*$, i.e. there exist open sets $U\supset N$
and $V\supset A^*$ such that $\cl(U)\cap V=\varnothing$, and thus
$\cl(U)\cap A^*=\varnothing$.
By \cite[Prop.\ 3.21]{KMV-1a} we have that $\omega(U) = A$, and therefore
\[
A = \omega(A) \subset \omega(N) \subset \omega(U) = A,
\]
which proves $\omega(N) = A$.
\end{proof}

\begin{lemma}
\label{injstunst}
The mapping $A\cap R \mapsto W^s(A\cap R)$ for an attractor $A$ and a repeller $R$ is injective.
\end{lemma}

\begin{proof}
 Let $U\in \sABlockC(\varphi)$ be an attracting block for $A$ so that  $U^c \in \sRBlockO(\varphi)$ is a repelling block for $A^*$.
Then, since $R$ is forward-backward invariant, $U\cap R$ is an attracting block in $R$, and $U^c\cap R$ is a repelling block in $R$.
From the properties of limit sets, cf.\ \cite[Lemma\ 2.9 and Propositions 2.11, 2.13]{KMV-1a}, since both $U$ and $R$ are forward invariant and
$U^c$ and $R$ are backward invariant, we have
\[
\omega(U\cap R) = \omega(U) \wedge \omega(R)= A\cap R\quad\hbox{and} \quad \alpha(U^c\cap R) = 
\alpha(U^c) \cap \alpha (R) = A^*\cap R.
\]
Therefore $(A\cap R,A^*\cap R)$ is an attractor-repeller pair in $R$.

Let $S = A\cap R$ and $S'= A'\cap R'$.
Suppose $W^s(S) = W^s(S')$. Then $W^s(S\wedge S') = W^s(S) = W^s(S')$ by Lemma~\ref{lem:stlatthom}.
Since $S\subset W^s(S)$ and $S'\subset W^s(S')$, we have 
\[
S\wedge S'\subset S\cup S'\subset  W^s(S\wedge S'),
\]
and $S\cup S'$ is compact.
Applying Lemma~\ref{aux} with $N=S\cup S'$, we have that $ \omega(S\cup S') = S\wedge S'$. Also $\omega(S\cup S')=\omega(S)\cup\omega(S')=S\cup S'$ by the invariance of $S,S'$. Therefore $S\cap S'\supset S\wedge S' =S\cup S'$ so that $S=S'$, which proves the injectivity.
\end{proof}

\subsection{Combinatorial Dynamics}
\label{sec:combdyn}

Let $\cX$ be a finite set.
A \emph{binary relation} $\cF$ on $\cX$ is subset of the product space $\cX\times \cX$.
We make use of the following concepts and structures, see \cite{KMV-1a,KMV-1b} for details.
In what follows $\cF$ can be interpreted as acting on sets by
\[
\cF(\cU) = \bigcup_{\xi\in \cU}\cF(\xi),\quad \cF(\xi) := \setof{\eta\in \cX \mid (\xi,\eta)\in \cF}.
\]

The  {\em forward invariant sets} and  {\em backward invariant sets} are given by
 \[
 \sInvset^+(\cF):=\{\cU\subset\cX \mid \cF(\cU)\subset\cU\}\;\hbox{and}\;
  \sInvset^-(\cF):=\{\cU\subset\cX \mid \cF^{-1}(\cU)\subset\cU\}.
 \]
These sets are sublattices of the Boolean algebra $\sSet(\cX)$, and the complement map $\cU\mapsto \cU^c$ is a lattice isomorphism from $\sInvset^+(\cF)$ to $\sInvset^-(\cF)$. 
A subset set $\cS\subset \cX$ is an {\em invariant set} if 
$\cS\subset \cF(\cS)$ and $ \cS\subset \cF^{-1}(\cS)$.
The invariant sets are denoted by $\IS(\cF)$, which is a lattice (not necessarily distributive).
As in the continuous case, $\Inv(\cU)$ denotes the maximal invariant set in $\cU$.

The sets of all {\em attractors} and {\em repellers} of $\cF$ are denoted by
\[
 \sAtt(\cF):=\{\cA\subset\cX~|~\cF(\cA)=\cA\}\;\text{and}\;\sRep(\cF):=\{\cR\subset\cX~|~\cF^{-1}(\cR)=\cR\},
\]
respectively, and are finite distributive lattices. Note that attractors and repellers are not necessarily invariant sets. 

The omega and alpha limit sets  in this setting are defined as follows
\begin{equation}\label{eqn:bomega}
\bomega(\cU) := \bigcap_{k\ge 0}\bigcup_{n\ge k}\cF^{n}(\cU);
\end{equation}
and
\begin{equation}\label{eqn:balpha}
 \balpha(\cU) :=  \bigcap_{k\le 0}\bigcup_{n\le k}\cF^{n}(\cU),
\end{equation}
which are forward and backward invariant sets respectively.

By \cite[Proposition~2.8]{KMV-1b}, $\bomega$ and $\balpha$ define surjective lattice homomorphisms  onto $\sAtt(\cF)$ and $\sRep(\cF)$ which yields the followinf commutative diagram
\begin{equation}
\label{eq:dualAlphaOmegaDiscrete}
\begin{diagram}
\node{\sInvset^+(\cF)} \arrow{e,l,<>}{^c}\arrow{s,l,A}{\bomega} \node{\sInvset^-(\cF)} \arrow{s,r,A}{\balpha} \\
\node{\sAtt(\cF)}     \arrow{e,l,<>}{*} \node{\sRep(\cF)} 
\end{diagram}
\end{equation}
where $\cA \mapsto \cA^*:=\balpha(\cA^c)$, cf.\ \cite[Diagram (5)]{KMV-1b}.

\subsection{Regular closed sets}
\label{sec:closedRegular}

For the purpose of relating combinatorial dynamics to topological dynamics it is useful to restrict the collection of sets used to discretize phase space.
Let $(X,\mathscr{T})$ be a topological space.
Define $U^{\rc\rc} := \cl\Int U$, then sets satisfying $U^{\rc\rc} = U$ are called the \emph{regular closed} sets in $\sSet(X)$ 
which  form a complete Boolean algebra $\scrR(X)$ under the operations
\[
U^\rc := \cl U^c,\quad
U\vee U' := U\cup U'
\quad\text{and}\quad 
 U\wedge U' := (U\cap U')^{\rc\rc} = \cl (\Int U\cap \Int U')
\]
cf.\ \cite{Walker}.

\begin{lemma}
\label{regclhom}
$^{\rc\rc}\colon \scrC(X) \to \scrR(X)$ given by $U\mapsto U^{\rc\rc}$ is a lattice homomorphism.
\end{lemma}

\begin{proof}
By definition $U\mapsto U^{\rc\rc}$ is an idempotent, order-preserving operator from $\sSet(X) \to \scrR(X)$.
A set $U$ is closed if and only if $U=\cl U$.
Let $\scrC(X)$ be the lattice of closed subsets $X$ which is a sublattice of  $\sSet(X)$.
Since $\Int U\subset U$, we have that $U^{\rc\rc} = \cl \Int U \subset \cl U =U$, which proves that $U\mapsto U^{\rc\rc}$ is also a
contractive operator. 
From the order-preserving property we have that $(U\cap U')^{\rc\rc} \subset U^{\rc\rc} \cap U'^{\rc\rc}$.
From all properties combined we have
\[
(U^{\rc\rc}\cap U'^{\rc\rc})^{\rc\rc} \subset (U\cap U')^{\rc\rc} = (U\cap U')^{\rc\rc\rc\rc}\subset 
(U^{\rc\rc}\cap U'^{\rc\rc})^{\rc\rc}
\]
which proves
\[
(U\cap U')^{\rc\rc} = (U^{\rc\rc}\cap U'^{\rc\rc})^{\rc\rc} = U^{\rc\rc}\wedge U'^{\rc\rc}.
\]
For unions
\[
(U\cup U')^{\rc\rc} = U^{\rc\rc} \cup U'^{\rc\rc},
\]
is proved in \cite[Sect.\ 4, Lem.\ 4]{Halmos} for regular open sets. The same statement for regular closed sets follows from duality $U\mapsto U^c$.
\end{proof}

For regular closed sets, the notion of `set-difference' is defined by
\begin{equation}
\label{defn:wedgeRegular}
U-U':= U\wedge U'^\rc
\end{equation}
Set-difference in $\scrR(X)$ can be related to set-difference in $\sSet(X)$.

\begin{lemma}
\label{lem:difference}
Let $U,U'\in \scrR(X)$. Then  $U- U' = \cl(U \smin U')$.
\end{lemma}

\begin{proof}
By definition $ U\wedge U'^\rc = \cl\Int(U\cap U'^\rc) = \cl\Int(U\cap \cl U'^c)$.
Since $U'$ is a regular closed set, the complement $U'^c$ is a regular open set, and therefore $\Int(\cl U'^c) = U'^c$. 
This yields
\[
\cl\Int(U\cap \cl U'^c) = \cl (\Int U \cap \Int\cl U'^c ) = \cl(\Int U \cap U'^c).
\]
Finally, since $U$ is a regular closed set, we have that  $\cl\Int U = U$ and thus $\cl(\Int U \cap U'^c) = \cl(U\cap U'^c)$, cf.\ \cite[pp.\ 35]{Walker}.
Combining this with the previous we obtain
\[
\cl\Int( U \cap U'^c) = \cl (U\cap U'^c) = \cl(U \smin U'),
\]
which proves the lemma.
\end{proof}

\vskip1cm

\end{sloppypar}
\end{document}